\documentclass{article}

\title{Bourgain's $L^2$ pointwise ergodic theorem over function fields}
\author{
Thái Hoàng Lê$^{3}$\and
Andrew Lott$^{1,2}$ 
}
\usepackage[margin=1in]{geometry} 
\usepackage{amsthm}
\usepackage{amsmath,amssymb,amsfonts,mathrsfs}
\usepackage{microtype}
\usepackage{indentfirst}
\usepackage{graphicx}
\usepackage{comment}
\usepackage{enumitem}
\usepackage{esint}
\usepackage[colorlinks=true,
            linkcolor=blue,
            urlcolor=blue,
            citecolor=blue]{hyperref}

\theoremstyle{definition} 
\newtheorem{thm}{Theorem}
\newtheorem{cor}{Corollary}[section]
\newtheorem{lem}{Lemma}[section]
\newtheorem{prop}{Proposition}[section]
\newtheorem{defn}{Definition}[section]

\newtheorem*{thmB'}{Theorem B$^\prime$}
\newtheorem*{thmC'}{Theorem C$^\prime$}
\newtheorem{externalthm}{Theorem}

\theoremstyle{remark}
\newtheorem*{remark}{Remark}

\numberwithin{equation}{section}

\newcommand{\F}{\mathbb{F}}

\newcommand{\N}{\mathbb{N}}
\newcommand{\Z}{\mathbb{Z}}
\newcommand{\1}{\textbf{1}}

\newcommand{\ord}{\text{ord}}
\newcommand{\wh}{\widehat}

\begin{document}

\maketitle

\begin{center}
$^{1}$University of Georgia \\
$^{2}$ HUN-REN Alfr\'ed R\'enyi Institute of Mathematics (Erd\H{o}s Center)\\
$^{3}$ The University of Mississippi
\end{center}

\begin{abstract}
We prove a function-field analogue of Bourgain's $L^2$ pointwise ergodic theorem. Let $q$ be a power of a prime $p$, let $\F_q[t]$ be the ring of polynomials over the finite field $\F_q$, and let $\F_q[t][u]$ be the ring of polynomials over $\F_q[t]$. Let $T^{(1)},\ldots,T^{(\ell)}$ be commuting, measure-preserving $\F_q[t]$-actions on a $\sigma$-finite measure space $(X,\mu)$, and let $P_1,\ldots,P_\ell\in \F_q[t][u]\setminus\{0\}$. Define a sequence of operators $(A_n)_{n\in \N}$ by 
\[
A_n g(x):=\frac{1}{q^n}\sum_{\substack{f\in \F_q[t]\\\deg f<n}}
g\left(T^{(1)}_{P_1(f)}\cdots T^{(\ell)}_{P_\ell(f)}x\right)
\qquad \left( g\in L^2(X),\,\,x\in X\right).
\]
We prove that $(A_n)_{n\in\N}$ satisfies an $L^2$ oscillation ergodic theorem: 
\[
\sup_{\substack{n_1<\cdots <n_{t_0}\\ t_0\in \N}}
\left(
\int_X
\sum_{j=1}^{t_0-1}
\sup_{n_j\leq n<n_{j+1}}
|A_ng(x)-A_{n_{j+1}}g(x)|^2
\,d\mu(x)
\right)^{1/2}
\leq C_1\|g\|_{L^2(X)}\qquad \left( g\in L^2(X)\right),
\]
where the constant $C_1>0$ depends only on $P_1,\ldots,P_\ell$ and $q$. This in particular implies that the sequence $(A_ng(x))_{n\in\N}$ converges for almost every $x\in X$ and that $(A_n)_{n\in\N}$ satisfies an $L^2$ maximal inequality:
\[
\left\|\sup_{n\in\N}|A_ng|\right\|_{L^2(X)}
\leq C_2\|g\|_{L^2(X)}
\qquad \left( g\in L^2(X)\right),
\]
where the constant $C_2>0$ depends only on $P_1,\ldots,P_\ell$ and $q$.
Our tools include the circle method in function fields and refinements of Weyl sum estimates in this setting, further developing the work of L\^e--Liu--Wooley and Champagne--Ge--L\^e--Liu--Wooley. These refinements are of independent interest.
\end{abstract}

\tableofcontents

\section{Introduction}

Bourgain's polynomial pointwise ergodic theorem is a landmark result in
ergodic theory. In the $L^2$ form most closely related to the present paper, it may
be stated as follows.

\begin{thm}[Bourgain]\label{thm:BourgainPolynomial}
Let $(X,\mu)$ be a $\sigma$-finite measure space. Let $T_1,\ldots,T_\ell:X\to X$ be invertible, commuting, measure-preserving transformations and 
$P_1,\ldots,P_\ell\in \Z[t]$.
If $g\in L^2(X)$, then the limit
\[
\lim_{N\to\infty}
\frac{1}{N}\sum_{n=1}^N
g\left(T_1^{P_1(n)}\cdots T_\ell^{P_\ell(n)}x\right)
\]
exists for almost every $x\in X$.
\end{thm}

Bourgain's proof combined ideas from number theory, harmonic analysis, probability, and the theory of Banach spaces, and his techniques remain central to the study of pointwise ergodic theorems \cite{Bourgain1988Maximal,Bourgain1988Lp,Bourgain1989ArithmeticSets}. See also \cite{Krause2023Bourbaki, Mirek2026CircleMethodPointwiseErgodic} for modern accounts of Bourgain's methods.

We prove an analogue of Bourgain's $L^2$ polynomial pointwise ergodic theorem in the function-field setting. Before we state the theorem, we must first establish definitions. 
An $\F_q[t]$-action $T$ on a $\sigma$-finite measure space $(X,\mu)$ is a family of measurable transformations
\[
\{T_a:X\to X\}_{a\in\F_q[t]}
\]
satisfying $T_0=\mathrm{id}$ and $T_{a+b}=T_aT_b$ for all $a,b\in\F_q[t]$. We say that such an
action is measure-preserving if, for every $a\in \F_q[t]$, the transformation $T_a\colon X\to X$ satisfies $\mu(E)=\mu(T_a^{-1}(E))$ for every measurable subset $E\subseteq X$. We say that $\F_q[t]$-actions $T^{(1)},\ldots,T^{(\ell)}$ on $(X,\mu)$ commute
if, for every $1\le i,j\le \ell$ and every $a,b\in \F_q[t]$, the transformations $T_a^{(i)}$ and
$T_b^{(j)}$ commute.

In fact, we prove the following quantitative result, which is known as a uniform oscillation inequality (conclusion $1$) and has a variety of consequences (conclusions $2$--$4$).

\begin{thm}\label{unweightedPW_0}
Let $T^{(1)},\ldots,T^{(\ell)}$ be commuting, measure-preserving $\F_q[t]$-actions on a $\sigma$-finite measure space $(X,\mu)$. Let $P_1,\ldots,P_\ell\in \F_q[t][u]\setminus\{0\}$. Define $A_n$ on $L^2(X)$ by
\[
A_ng(x):=\frac{1}{q^n}\sum_{\deg f<n}
g\big(T^{(1)}_{P_1(f)}\cdots T^{(\ell)}_{P_\ell(f)}x\big).
\]
Then there exist constants $C_1,C_2>0$ depending only on $P_1,\ldots,P_\ell$ and $q$ such that the following hold.
\begin{enumerate}[label=(\arabic*)]
\item For every $g\in L^2(X)$,
\[
\sup_{\substack{n_1<\cdots <n_{t_0}\\ t_0\in \N}}
\left(
\int_X
\sum_{j=1}^{t_0-1}
\sup_{n_j\leq n<n_{j+1}}
|A_ng(x)-A_{n_{j+1}}g(x)|^2
\,d\mu(x)
\right)^{1/2}
\leq C_1\|g\|_{L^2(X)}.
\]
\item For every $g\in L^2(X)$,
\[
\left\|\sup_{n\in\N}|A_ng|\right\|_{L^2(X)}
\leq C_2\|g\|_{L^2(X)}.
\]
\item For every $g\in L^2(X)$, the sequence $(A_ng(x))_{n\in\N}$ converges for almost every $x\in X$.
\item For every $g\in L^2(X)$, the sequence $(A_ng)_{n\in\N}$ converges with respect to the $L^2$-norm on $X$.
\end{enumerate}
\end{thm}

Observe that conclusion ($1$) implies conclusions ($2$) and ($3$); see Mirek--Szarek--Wright for details \cite{MirekSzarekWright2022Oscillation}. Conclusions ($2$) and ($3$) then imply conclusion ($4$) by the dominated convergence theorem. Since the averages in our case are naturally normalized by a lacunary sequence $(q^{-n})_{n\in\N}$, we are able to achieve an $L^2$ oscillation inequality which is analogous to the strongest $L^2$ oscillation inequality known over the integers; see Mirek--S{\l}omian--Szarek \cite{MirekSlomianSzarek2023OscillationRemarks}. We adapt the method of Quas--Wierdl \cite{Bergelson2006}, which also contains ideas from Bourgain and Lacey \cite{Lacey1997BourgainInequality}.

Our manuscript builds on the recent work of L\^e--Liu--Wooley \cite{LeLiuWooley2025} and Champagne--Ge--L\^e--Liu--Wooley \cite{ChampagneGeLeLiuWooley2025} towards quantitative Weyl sum estimates over function fields. In particular, we prove an $\varepsilon$-free version of the Weyl estimate of L\^e--Liu--Wooley \cite{LeLiuWooley2025}, and we use this to derive a sufficiently general major arc estimate. This $\varepsilon$-free estimate is of independent interest and has potential applications on problems related to exponential sums on $\F_q[t]$.

\section{Acknowledgments}
The second author was supported by the HUN-REN Alfr\'ed R\'enyi Institute of Mathematics (Erd\H{o}s Center) during the completion of this manuscript, and he thanks the R\'enyi Institute for its hospitality. The second author would also like to thank Tomasz Szarek for introducing him to pointwise ergodic theorems. The first author is supported by NSF Grant DMS-2246921 and a Travel Support for Mathematicians gift from the Simons Foundation.

\section{Notation}\label{notation}

We collect the notation used throughout the paper.
\begin{itemize}
\item $\N$ denotes the set of positive integers.
\item Let $q=p^m$, where $p$ is prime, and let
$\F_q$ denote the finite field with $q$ elements. We write $\F_q[t]$ for the polynomial ring over
$\F_q$, and we use the convention
\[
\deg 0=-\infty.
\]
Thus, for $N\ge 0$, the set
\[
\{f\in\F_q[t]:\deg f<N\}
\]
has cardinality $q^N$. All sums of the form $\sum_{\deg f<N}$ are taken over $f\in\F_q[t]$.

\item Let
\[
\mathbb K:=\F_q(t)
\qquad\text{and}\qquad
\mathbb K_\infty:=\F_q((1/t)).
\]
Thus every nonzero $\alpha\in\mathbb K_\infty$ has a unique Laurent expansion
\[
\alpha=\sum_{j\le J} a_jt^j,
\qquad a_J\neq 0,\quad a_j\in\F_q,
\]
and we define
\[
\ord \alpha:=J,
\qquad
\ord 0:=-\infty.
\]

\item We write
\[
\mathbb T:=\left\{\sum_{j\le -1}a_jt^j:a_j\in\F_q\right\}.
\]

\item Let $\operatorname{Tr}_{\F_q/\F_p}$ denote the field trace. If
\[
\alpha=\sum_{j\le J}a_jt^j\in\mathbb K_\infty,
\]
we write
\[
\operatorname{res}\alpha:=a_{-1},
\]
with the convention that $a_{-1}=0$ if no $t^{-1}$ term appears. Define the additive character
$e:\mathbb K_\infty\to\mathbb C^\times$ by
\[
e(\alpha):=\exp\left(\frac{2\pi i}{p}\operatorname{Tr}_{\F_q/\F_p}(\operatorname{res}\alpha)\right).
\]

Observe $e$ is trivial on $\F_q[t]$. In particular, if $\ord\{\alpha\}<-1$, then $e(\alpha)=1$. We shall repeatedly use the following orthogonality relation. For every $N\ge 0$ and every
$\alpha\in\mathbb K_\infty$,
\begin{equation}\label{eq:orthogonality}
\sum_{\deg f<n} e(f\alpha)
=
q^n\mathbf 1_{\ord\{\alpha\}<-n}.
\end{equation}

\item We write bold letters for vectors. Thus
\[
\boldsymbol{x}=(x_1,\ldots,x_k)\in\F_q[t]^k,
\qquad
\boldsymbol{\alpha}=(\alpha_1,\ldots,\alpha_k)\in\mathbb T^k,
\]
and
\[
\boldsymbol{x}\cdot\boldsymbol{\alpha}
:=
x_1\alpha_1+\cdots+x_k\alpha_k.
\]
If $\boldsymbol{a}=(a_1,\ldots,a_k)$ and $h\in\F_q[t]\setminus\{0\}$, then
\[
\frac{\boldsymbol{a}}{h}
:=
\left(\frac{a_1}{h},\ldots,\frac{a_k}{h}\right).
\]

\item For a finitely supported function $g:\F_q[t]^k\to\mathbb C$, we define its Fourier transform by
\[
\widehat g(\boldsymbol{\alpha})
:=
\sum_{\boldsymbol{x}\in\F_q[t]^k}
g(\boldsymbol{x})e(-\boldsymbol{x}\cdot\boldsymbol{\alpha}),
\qquad
\boldsymbol{\alpha}\in\mathbb T^k.
\]
Then Fourier inversion and Plancherel take the forms
\begin{equation}\label{eq:fourier-inversion}
g(\boldsymbol{x})
=
\int_{\mathbb T^k}
\widehat g(\boldsymbol{\alpha})
e(\boldsymbol{x}\cdot\boldsymbol{\alpha})
\,dm(\boldsymbol{\alpha}),
\end{equation}
and
\begin{equation}\label{eq:plancherel}
\sum_{\boldsymbol{x}\in\F_q[t]^k}|g(\boldsymbol{x})|^2
=
\int_{\mathbb T^k}|\widehat g(\boldsymbol{\alpha})|^2\,dm(\boldsymbol{\alpha}),
\end{equation}
where $m$ is the normalized Haar measure on $\mathbb T^k$. These identities extend in the usual way to $\ell^2(\F_q[t]^k)$. An operator $T$ on $\ell^2(\F_q[t]^k)$ is a Fourier multiplier when
\[
\widehat{Tg}(\boldsymbol{\alpha})
=
\widehat T(\boldsymbol{\alpha})\widehat g(\boldsymbol{\alpha}).
\]

\item If $T$ is a measure-preserving transformation of a measure space $(X, \mu)$, and $g$ is a measurable function on $X$, then $Tg$ is a function on $X$ defined by
\[
(Tg)(x) = g(Tx)
\]
for every $x \in X$.

\item If $a_1,\ldots,a_k,h\in\F_q[t]$ with $h\neq 0$, then
\[
(a_1,\ldots,a_k,h)
\]
denotes their monic greatest common divisor. Thus $(a_1,\ldots,a_k,h)=1$ means that these
polynomials have no common nonconstant divisor. Least common multiples are also taken to be
monic. Unless stated otherwise, rational vectors $\boldsymbol a/h$ are written with $h$ monic and
$\deg a_i<\deg h$ for every $i$.

\item We use Vinogradov notation in the standard way. Thus $A\ll B$ means that $|A|\le CB$ for some
constant $C>0$. Subscripts indicate the allowed dependence of the implicit constant; for instance,
$A\ll_{\mathcal K,q}B$ means that $C$ may depend on $\mathcal K$ and $q$.

\item Lastly, we use $\mathbf 1_E$ for the indicator function of a set $E$. We also use the shorthand
$\mathbf 1_{\mathcal P}$ for the indicator of a condition $\mathcal P$. 
\end{itemize}
\section{Sketch of proof}
The structure of our proof is inspired by Quas-Wierdl \cite[Appendix~B]{Bergelson2006}. To begin, it is convenient to use the Frobenius property to rewrite the iterates $T_{P_1(f)}^{(1)}\cdots T_{P_\ell(f)}^{(\ell)}$ in a normal form $S_{f^{r_1}}^{(1)}\cdots S_{f^{r_k}}^{(k)}$, where $(r_i,p)=1$ for $1\leq i\leq k$. Recall that the goal is to show that the sequence of operators $(A_n)_{n\in \N}$ defined by 
\[
A_ng(x):=\frac{1}{q^n}\sum_{\deg f<n}g(S_{f^{r_1}}^{(1)}\cdots S_{f^{r_k}}^{(k)}x)
\]
satisfies a uniform oscillation inequality, which in our case takes the form 
\[
\sup_{\substack{n_1<\cdots <n_{t_0}\\ t_0\in \N}}
\left(
\int_X
\sum_{j=1}^{t_0-1}
\sup_{n_j\leq n<n_{j+1}}
|A_ng(x)-A_{n_{j+1}}g(x)|^2
\,d\mu(x)
\right)^{1/2}
\leq C\|g\|_{L^2(X)},
\]
where $C$ is a constant depending on $r_1,...,r_k$ and $q$. This implies that $\lim_{n\to \infty} A_n f(x)$ exists for almost every $x\in X$. 

We use a transference argument to show that it suffices to prove the oscillation inequality for the operator $M_n$ on $\ell^2(\F_q[t]^k)$ defined by 
\[
M_n g(x_1,...,x_k)=\frac{1}{q^n}\sum_{\deg f<n}g(x_1+f^{r_1},...,x_k+f^{r_k}).
\]
A critical observation here is that the sum above is in the form of a convolution, and thus it is not hard to see that $M_n$ is a Fourier multiplier so that in particular 
\[
\widehat{M_ng}(\alpha_1,...,\alpha_k)=\widehat{M_n}(\alpha_1,...,\alpha_k)\widehat{g}(\alpha_1,...,\alpha_k).
\]
Additionally, $\widehat{M_n}$ is a polynomial exponential sum over $\F_q[t]$: 
\[
\widehat{M_n}(\alpha_1,...,\alpha_k)=\frac{1}{q^n}\sum_{\deg f<n}e(\alpha_1f^{r_1}+\cdots+\alpha_k f^{r_k}).
\]
At this point we note that because each $r_i$ is coprime to $p$, the exponential sum estimates of Lê-Liu-Wooley \cite{LeLiuWooley2025} apply cleanly to the sum above. This is the advantage of the reduction to the normal form at the beginning of the argument.  

Armed with the exponential sum estimates of \cite{LeLiuWooley2025}, we prove that the sequence of operators $(M_n)$ is equivalent to a sequence of operators $(C_n)$ which we call the major arc model. (See Definition \ref{operatorequivalence} for the definition of equivalence.) In particular, $C_n$ is defined by 
\[
\widehat{C}_n(\alpha_1,...,\alpha_k)
=
\sum_{\substack{h\ \mathrm{monic}\\ \deg h<\rho n}}
\sum_{\substack{a_1,...,a_k\in\F_q[t]\\ \deg a_i<\deg h\ \forall i\\(a_1,...,a_k,h)=1}}
\Lambda(a_1,...,a_k,h)
\1_{\ord(\alpha_i-a_i/h)<-r_i n\ \forall 1\leq i\leq k},
\]
where $\Lambda(a_1,...,a_k,h)$ is a Gauss sum and $\rho$ is a suitable constant. Notice that $\widehat{C_n}$ is supported on very narrow boxes. We achieve this through a pruning procedure, for which we required a sufficiently strong major arc estimate for $\widehat{M_n}$. Over the integers this estimate is standard, but the analogous estimate did not exist in the function field setting at the required level of generality. In the appendix, we achieve this major arc estimate by repeated applications of the Weyl estimate \cite[Theorem~3.1]{LeLiuWooley2025}.

In addition, our notion of equivalence implies that if $(C_n)$ satisfies the oscillation inequality, then $(M_n)$ satisfies the oscillation inequality as well. Thus, the problem reduces to showing $(C_n)$ satisfies the oscillation inequality. The strategy now is to organize the sum defining $\wh{C}_n$ based on the degree of $h$, and define operators $C_{s,n}$ given by 
\[
\widehat{C}_{s,n}(\alpha_1,...,\alpha_k)
=
\1_{s<\rho n}
\sum_{\substack{h\ \mathrm{monic}\\ \deg h=s}}
\sum_{\substack{a_1,...,a_k\in\F_q[t]\\ \deg a_i<\deg h\ \forall i\\(a_1,...,a_k,h)=1}}
\Lambda(a_1,...,a_k,h)
\1_{\ord(\alpha_i-a_i/h)<-r_i n\ \forall 1\leq i\leq k}.
\]
For a fixed $s$, the sequence $(C_{s,n})$ is almost a monotone sequence of projections. In particular, for fixed $s$, there is a monotone sequence of projections $(D_{s,n})$ on $\ell^2(\F_q[t]^k)$ such that for any $g\in \ell^2(\F_q[t]^k)$, there is a $w_s\in \ell^2(\F_q[t]^k)$ satisfying
\[
C_{s,n}g=D_{s,n}w_s
\]
for every $n$, with
\[
\|w_s\|_{\ell^2(\F_q[t]^k)}\ll q^{-\gamma s}\|g\|_{\ell^2(\F_q[t]^k)}.
\]
With this in mind, it suffices to show that $D_{s,n}$ satisfies an oscillation inequality with at most a logarithmic loss. Here we also use the fact that we have an exponential bound on the Gauss sum:
\[
|\Lambda(a_1,...,a_k,h)|\ll q^{-\gamma s}
\qquad (\deg h=s).
\]

Finally, since $D_{s,n}$ is a monotone sequence of projections, proving the oscillation inequality for $D_{s,n}$ requires that $D_{s,n}$ satisfies a maximal inequality of the form 
\[
\|D_s^{\ast}g\|_{\ell^2(\F_q[t]^k)}
\ll \log(2+R_s)\|g\|_{\ell^2(\F_q[t]^k)},
\]
where
\[
R_s=\deg Q_s,
\qquad
Q_s=\prod_{\substack{h\ \mathrm{monic}\\\deg h =s}}h.
\]
We achieve this by cases based on the size of $n$ (the so-called small scales and large scales). We handle the small scales by a simple Rademacher-Menshov type argument. On the other hand, the argument over the large scales quickly reduces to linear ergodic maximal estimate, which we prove via a Vitali-type covering argument and Marcinkiewicz interpolation.

\section{Reduction to the normal form}
Our first goal is to replace arbitrary polynomials appearing in the statement of Theorem \ref{unweightedPW_0} by monomials whose exponents are coprime to the characteristic. This is the natural normal form for the argument that follows, and in particular it is the form to which the exponential sum estimates apply cleanly. The idea behind this reduction is inspired by \cite[Section~3]{ChampagneGeLeLiuWooley2025}. 

\begin{lem}\label{preliminaryreduction}
Suppose $T^{(1)},...,T^{(\ell)}$ are commuting, measure preserving $\F_q[t]$-actions on a measure space $(X,\mu)$, and let $P_1,...,P_\ell\in \F_q[t][u]\setminus\{0\}$ satisfy $P_j(0)=0$ for every $1\le j\le \ell$. Then there exist an integer $k\geq 1$, integers $1\le r_1<\cdots<r_k$ satisfying $(r_i,p)=1$ for every $1\le i\le k$, and commuting measure preserving $\F_q[t]$-actions $S^{(1)},...,S^{(k)}$ on $(X,\mu)$ such that for every $f\in \F_q[t]$,
\begin{equation}\label{CoprimeNested}
T^{(1)}_{P_1(f)}\cdots T^{(\ell)}_{P_\ell(f)}=S_{f^{r_1}}^{(1)}\cdots S_{f^{r_k}}^{(k)}.
\end{equation}
\end{lem}
\begin{proof}
For each $1\le j\le \ell$, write
\[
P_j(u)=\sum_{\nu=1}^{N_j} a_{j,\nu}u^{e_{j,\nu}},
\]
where $N_j\ge 1$, $a_{j,\nu}\in \F_q[t]$, and $e_{j,\nu}\ge 1$. Since each $T^{(j)}$ is an $\F_q[t]$-action, we have
\[
T^{(j)}_{P_j(f)}=\prod_{\nu=1}^{N_j}T^{(j)}_{a_{j,\nu}f^{e_{j,\nu}}}
\]
for every $1\le j\le \ell$. Using also that the actions $T^{(1)},\ldots,T^{(\ell)}$ commute, we may reorder all of the monomial factors and group together those whose exponents have the same $p$-free part. Thus it is enough to consider a subproduct of the form
\[
T^{(j_1)}_{a_1f^{e_1}}\cdots T^{(j_m)}_{a_mf^{e_m}},
\]
where $a_1,\ldots,a_m\in \F_q[t]$ and $e_i=p^{n_i}r$ for every $1\le i\le m$, for some fixed integer $r$ coprime to $p$. Since $\mathbb{F}_q$ has characteristic $p$, each map $u\mapsto a_i u^{p^{n_i}}$
is additive on $\F_q[t]$ by the Frobenius property. We therefore define
\[
S_u:=T^{(j_1)}_{a_1u^{p^{n_1}}}\cdots T^{(j_m)}_{a_mu^{p^{n_m}}}.
\]
Because the original actions commute, $S$ is an $\F_q[t]$-action; indeed, for $u,v\in \F_q[t]$,
using $(u+v)^{p^{n_i}}=u^{p^{n_i}}+v^{p^{n_i}}$ for every $i$ and then reordering commuting factors, we get $S_{u+v}=S_uS_v.$
Moreover, each $S_u$ is a composition of measure preserving transformations, and is therefore itself measure preserving. Finally, by design we have 
\[
S_{f^r}=T^{(j_1)}_{a_1f^{e_1}}\cdots T^{(j_m)}_{a_mf^{e_m}}.
\]
There are only finitely many distinct $p$-free parts occurring among the exponents in $P_1,\ldots,P_\ell$; enumerate them as $1\le r_1<\cdots<r_k$. Performing the above construction for each of them yields commuting measure preserving $\F_q[t]$-actions $S^{(1)},...,S^{(k)}$ and hence the representation \eqref{CoprimeNested}. Moreover, if we have $S^{(i)}$, $S^{(j)}$ and $u,v\in \F_{q}[t]$, then every factor in $S^{(i)}_u$ commutes with every factor in $S^{(j)}_v$, by the commutativity hypothesis on the original $\F_q[t]$-actions. Hence the constructed $\F_q[t]$-actions $S^{(1)},...,S^{(k)}$ commute.
\end{proof}

Combining Lemma~\ref{preliminaryreduction} with the harmless removal of constant terms explained below, Theorem~\ref{unweightedPW_0} reduces to the following.
\begin{thm}\label{unweightedPW}
Let $S^{(1)},...,S^{(k)}$ be commuting measure preserving $\F_q[t]$-actions on a measure space $(X,\mu)$. Let $r_1<r_2<\cdots<r_k$ be positive integers coprime to $p$. Let $g\in L^2(X)$. Then,
\[\lim_{n\to\infty} \frac{1}{q^n}\sum_{
\deg f<n}g\big(S_{f^{r_1}}^{(1)}\cdots S_{f^{r_k}}^{(k)}x\big)
\;
\text{ exists for almost every } x\in X.\]
\end{thm}
\begin{remark}
Notice that to apply Proposition \ref{preliminaryreduction} we require the polynomials to each have $0$ constant term. This turns out to be harmless for the following reason. Suppose $P_j\in \F_q[t][u]$ for $1\le j\le \ell$ are the polynomials appearing in Theorem \ref{unweightedPW_0}, and write $P_j(u)=Q_j(u)+a_j$,
with $Q_j(0)=0$ and $a_j\in \F_q[t]$. We write
\[
\{j\in\{1,\ldots,\ell\}:Q_j\neq 0\}=\{j_1<\cdots<j_m\}.
\]
Define $g_0:X\to \mathbb{C}$ by 
\[
g_0(x):=g\big(T^{(1)}_{a_1}\cdots T^{(\ell)}_{a_\ell}x\big).
\]
Since the actions are commuting and measure preserving, we have $g_0\in L^2(X)$ and
\[
g\big(T^{(1)}_{P_1(f)}\cdots T^{(\ell)}_{P_\ell(f)}x\big)
=g_0\big(T^{(j_1)}_{Q_{j_1}(f)}\cdots T^{(j_m)}_{Q_{j_m}(f)}x\big),
\]
where the right-hand side is interpreted as $g_0(x)$ when $m=0$. If $m=0$, then the averages are identically equal to $g_0(x)$ and there is nothing to prove. Otherwise, the constant terms may be absorbed into $g$, the indices with $Q_j=0$ may be discarded, and it is enough to treat the case of nonzero polynomials with zero constant term covered by Lemma \ref{preliminaryreduction}.
\end{remark}
\section{The oscillation inequality}
Given a sequence $(a_n)_{n\in\N}$ of complex numbers and positive integers
$n_1<\cdots<n_{t_0}$, define the oscillation functional $\mathcal{O}_{n_1,\ldots,n_{t_0}}$ by 
\begin{equation}\label{osc}
\mathcal{O}_{n_1,\ldots,n_{t_0}}(a_n)
:=
\left(
\sum_{j=1}^{t_0-1}
\sup_{n_j\le n < n_{j+1}}
\big|a_n-a_{n_{j+1}} \big|^2
\right)^{1/2}.
\end{equation}
The oscillation functional is a standard object in the study of pointwise ergodic theorems. We list the known facts and definitions that we will need. We omit the proofs. (See \cite{MirekSzarekWright2022Oscillation} for details).

\begin{defn}
We say that a sequence of operators $(A_n)$ on $L^2(X,\mu)$ satisfies the \textit{oscillation inequality} if there exists a constant $C>0$ such that, for every $f\in L^2(X)$ and every finite increasing sequence
$n_1<\cdots<n_{t_0}$ of positive integers, one has
\[
\|\mathcal{O}_{n_1,\ldots,n_{t_0}}(A_n f)\|_{L^2(X)}\le C\|f\|_{L^2(X)}.
\]
\end{defn}

\begin{lem}\label{oscImpliesConv}
Suppose $(A_n)$ is a sequence of operators on $L^2(X)$ which satisfies the oscillation inequality. Then $A_nf(x)$ converges as $n\to\infty$ for almost every $x\in X$. 
\end{lem}

The oscillation functional also satisfies a triangle inequality. In other words, $\mathcal{O}_{n_1,\ldots,n_{t_0}}$ is a subadditive functional.

\begin{lem}\label{oscTriangle}
For any two complex sequences $(a_n)$ and $(b_n)$ and any finite increasing sequence of positive integers $n_1<\cdots<n_{t_0}$, one has
\[
\mathcal{O}_{n_1,\ldots,n_{t_0}}(a_n+b_n)
\leq \mathcal{O}_{n_1,\ldots,n_{t_0}}(a_n)+\mathcal{O}_{n_1,\ldots,n_{t_0}}(b_n).
\]
\end{lem}

The next notion will let us replace one operator sequence by another without changing the
validity of an oscillation inequality.

\begin{defn}\label{operatorequivalence}
Let $(X,\mu)$ be a measure space, and let $(A_n)_{n\in\N}$ and $(B_n)_{n\in\N}$ be sequences of
linear operators on $L^2(X,\mu)$. We say that $(A_n)$ and $(B_n)$ are \emph{equivalent on $L^2(X,\mu)$} if, for every $f\in L^2(X,\mu)$,
\[
\sum_{n=1}^\infty \|A_nf-B_nf\|_{L^2(X,\mu)}^2\ll \|f\|_{L^2(X,\mu)}^2.
\]
\end{defn}

\begin{lem}\label{equivalenceOscLemma}
Let $(X,\mu)$ be a $\sigma$-finite measure space, and let $(A_n)_{n\in\N}$ and $(B_n)_{n\in\N}$ be equivalent
sequences of linear operators on $L^2(X,\mu)$. If $(B_n)$ satisfies the oscillation inequality, then $(A_n)$ also satisfies the oscillation inequality. 
\end{lem}

In the course of the proof of Theorem \ref{unweightedPW}, we will prove the following.

\begin{thm}\label{oscillationinequality}
Let $r_1,\ldots,r_k$ be positive integers with $(r_i,p)=1$ for every $1\leq i\leq k$ and let $S^{(1)},\ldots,S^{(k)}$ be commuting, measure-preserving $\F_q[t]$-actions on a $\sigma$-finite measure space $(X,\mu)$. Define an operator $A_n$ on $L^2(X,\mu)$ by 
\[
A_ng(x)=\frac{1}{q^n}\sum_{\deg u<n}
g\left(S^{(1)}_{u^{r_1}}\cdots S^{(k)}_{u^{r_k}}x\right).
\]
Then $(A_n)$ satisfies the oscillation inequality. 
\end{thm}

\section{A Calder\'on transference argument}

By a direct transference argument, it suffices to prove an analogous
oscillation inequality for a discrete model on $\F_q[t]$ with counting measure.
For $n\in \N$ and $r_1<\cdots<r_k$ with $(r_j,p)=1$ for every $j$, define
\begin{equation}\label{Mn}
M_n^{(r_1,\ldots,r_k)} g(a_1,\ldots,a_k)
=
\frac{1}{q^n}\sum_{\deg u< n}
g(a_1+u^{r_1},\ldots,a_k+u^{r_k}),
\qquad g\in \ell^2(\F_q[t]^k),\ a_1,\ldots,a_k\in \F_q[t].
\end{equation}

\begin{thm}\label{discreteOsc}
Let $g\in \ell^2(\F_q[t]^k)$, and suppose $r_1,\ldots,r_k$ are distinct positive integers with $(r_i,p)=1$ for every $1\leq i\leq k$. Then $M_n^{(r_1,\ldots,r_k)}$ satisfies the oscillation inequality. 
\end{thm}

\begin{proof}[Proof of Theorem \ref{unweightedPW} assuming Theorem \ref{discreteOsc}]
Define, for $x\in X$ and a large positive integer $K$,
\begin{equation}\label{Phi}
\Phi_{x,K}(a_1,\ldots,a_k)
:=
\1_{\deg a_1<r_1K}\cdots\1_{\deg a_k<r_kK}
\bigl(S^{(1)}_{a_1}\cdots S^{(k)}_{a_k}g\bigr)(x),
\qquad a_1,\ldots,a_k\in \F_q[t].
\end{equation}
If $n\leq K$ and $\deg a_i<r_iK$ for every $1\leq i\leq k$, then
\begin{equation}\label{eq:transfer}
A_n\bigl(S^{(1)}_{a_1}\cdots S^{(k)}_{a_k}g\bigr)(x)
=
M_n^{(r_1,\ldots,r_k)}\Phi_{x,K}(a_1,\ldots,a_k).
\end{equation}
Indeed,
\begin{align*}
M_n^{(r_1,\ldots,r_k)}\Phi_{x,K}(a_1,\ldots,a_k)
&=\frac{1}{q^n}\sum_{\deg u<n}\Phi_{x,K}(a_1+u^{r_1},\ldots,a_k+u^{r_k})\\
&=\frac{1}{q^n}\sum_{\deg u<n}
\bigl(S^{(1)}_{a_1+u^{r_1}}\cdots S^{(k)}_{a_k+u^{r_k}}g\bigr)(x)\\
&=\frac{1}{q^n}\sum_{\deg u<n}
\bigl(S^{(1)}_{u^{r_1}}\cdots S^{(k)}_{u^{r_k}}
S^{(1)}_{a_1}\cdots S^{(k)}_{a_k}g\bigr)(x)\\
&=A_n\bigl(S^{(1)}_{a_1}\cdots S^{(k)}_{a_k}g\bigr)(x),
\end{align*}
where in the second line we used the degree assumptions to guarantee that the indicators in \eqref{Phi} are equal to $1$.

Now fix positive integers $n_1<\cdots<n_{t_0}\le K$. Since the actions commute with $A_n$ and each
$S^{(j)}_a$ is measure preserving, we have
\begin{align*}
q^{K(r_1+\cdots+r_k)}
\|\mathcal{O}_{n_1,\ldots,n_{t_0}}(A_n g)\|_{L^2(X)}^2
&=
\sum_{\deg a_1< r_1K}\cdots\sum_{\deg a_k< r_kK}
\|\mathcal{O}_{n_1,\ldots,n_{t_0}}(A_n(S^{(1)}_{a_1}\cdots S^{(k)}_{a_k} g))\|_{L^2(X)}^2\\
&=
\sum_{\deg a_1< r_1K}\cdots\sum_{\deg a_k< r_kK}
\int_X
\bigl|\mathcal{O}_{n_1,\ldots,n_{t_0}}(A_n(S^{(1)}_{a_1}\cdots S^{(k)}_{a_k} g)(x))\bigr|^2\,d\mu(x).
\end{align*}
Using \eqref{eq:transfer}, this is at most
\begin{align*}
&\int_X
\sum_{(a_1,\ldots,a_k)\in \F_q[t]^k}
\bigl|\mathcal{O}_{n_1,\ldots,n_{t_0}}
(M_n^{(r_1,\ldots,r_k)}\Phi_{x,K}(a_1,\ldots,a_k))\bigr|^2\,d\mu(x)\\
&=
\int_X
\|\mathcal{O}_{n_1,\ldots,n_{t_0}}
(M_n^{(r_1,\ldots,r_k)}\Phi_{x,K})\|_{\ell^2(\F_q[t]^k)}^2\,d\mu(x).
\end{align*}
By Theorem \ref{discreteOsc}, we obtain
\[
\int_X
\|\mathcal{O}_{n_1,\ldots,n_{t_0}}
(M_n^{(r_1,\ldots,r_k)}\Phi_{x,K})\|_{\ell^2(\F_q[t]^k)}^2\,d\mu(x)
\ll
\int_X \|\Phi_{x,K}\|_{\ell^2(\F_q[t]^k)}^2\,d\mu(x).
\]
Finally,
\begin{align*}
\int_X \|\Phi_{x,K}\|_{\ell^2(\F_q[t]^k)}^2\,d\mu(x)
&=
\int_X
\sum_{\deg a_1<r_1K}\cdots\sum_{\deg a_k<r_kK}
|(S^{(1)}_{a_1}\cdots S^{(k)}_{a_k} g)(x)|^2\,d\mu(x)\\
&=
\sum_{\deg a_1<r_1K}\cdots\sum_{\deg a_k<r_kK}
\|g\|_{L^2(X)}^2\\
&=
q^{K(r_1+\cdots+r_k)}\|g\|_{L^2(X)}^2.
\end{align*}
Putting all of this together and dividing through by $q^{K(r_1+\cdots+r_k)}$ yields
\[
\|\mathcal{O}_{n_1,\ldots,n_{t_0}}(A_ng)\|_{L^2(X)}^2
\ll
\|g\|_{L^2(X)}^2
\]
for every finite increasing sequence $n_1<\cdots<n_{t_0}\le K$. Since $K$ is arbitrary, this proves
Theorem \ref{oscillationinequality}. Lemma \ref{oscImpliesConv} then gives Theorem
\ref{unweightedPW}.
\end{proof}

\section{The circle method}

Consider the discrete shift averaging operator $M^{(r_1,...,r_k)}_n$ as defined in the previous section, and from now on suppress the dependence on $r_1,...,r_k$ by writing $M_n= M^{(r_1,...,r_k)}_n$. The great advantage of working with a discrete shift average $M_n$ is that it is a Fourier multiplier operator where $\wh{M}_n$ is a polynomial exponential sum that can be analyzed via the circle method. In particular, by a change of variables, for each $\boldsymbol{\alpha}=(\alpha_1,...,\alpha_k)\in \mathbb{T}^k$, we have 
\begin{align*}
\widehat{M_ng}(\boldsymbol{\alpha})&=\sum_{\boldsymbol{u}\in \F_q[t]^k}\frac{1}{q^n}\sum_{\deg f<n}g(\boldsymbol{u}+(f^{r_1},...,f^{r_k}))e(-\boldsymbol{u}\cdot \boldsymbol{\alpha}) \\
&=\frac{1}{q^n}\sum_{\deg f<n }e(\alpha_1f^{r_1}+\cdots+\alpha_kf^{r_k})\sum_{\boldsymbol{u}\in\F_q[t]^k}g(\boldsymbol{u})e(-\boldsymbol{u}\cdot \boldsymbol{\alpha}) \\
&=\widehat{M}_n(\boldsymbol{\alpha})\widehat{g}(\boldsymbol{\alpha}),
\end{align*}
where
\[\widehat{M}_n(\alpha_1,...,\alpha_k)=\frac{1}{q^n}\sum_{\deg f<n}e(\alpha_1f^{r_1}+\cdots +\alpha_kf^{r_k}).\]

Throughout this section we assume $\mathcal{K}=\{r_1,...,r_k\}$ with $(r_i,p)=1$ for every $1\leq i\leq k$. Let $c,C>0$ be the constants obtained when applying Corollary \ref{epsilonfree43} with
$\mathcal{K}$. Set
\begin{equation}\label{circlemethodparameters}
r^\ast:=\max\{r_1,\ldots,r_k\}, \qquad
\rho:=\frac{1}{8r^\ast}, \qquad
\delta_0:=\min\!\left\{\frac c2,\frac{1}{32C(r^\ast)^2}\right\}.
\end{equation}

Define
\begin{equation}\label{majorarcbox}
N_n(a_1,\ldots,a_k,h)
:=
\Big\{(\alpha_1,\ldots,\alpha_k)\in \mathbb{T}^k:
\ord\!\Big(\alpha_i-\frac{a_i}{h}\Big)<-r_i n+\frac{n}{4(r^{\ast})^2}
\ \forall\,1\le i\le k\Big\}.
\end{equation}
We then define the \emph{major arcs} by
\begin{equation}\label{majorarcs}
\mathfrak{M}_n
=\bigcup_{\substack{h\ \mathrm{monic}\\ \deg h<\rho n}}
\ \ \bigcup_{\substack{a_1,\ldots,a_k\in \F_q[t]\\ \deg a_i<\deg h\ \forall i\\ (a_1,\ldots,a_k,h)=1}}
N_n(a_1,\ldots,a_k,h),
\end{equation}
and the \emph{minor arcs} by$
\mathfrak{m}_n := \mathbb{T}^k\setminus \mathfrak{M}_n.$

We will also use that these major arcs are pairwise disjoint. Indeed, suppose that
\[
\boldsymbol{\alpha}\in N_n(a_1,...,a_k,h)\cap N_n(b_1,...,b_k,g),
\]
where $h,g$ are monic, $\deg h,\deg g<\rho n$, and both fractions are reduced in the sense that
$(a_1,...,a_k,h)=(b_1,...,b_k,g)=1$. If $\boldsymbol{a}/h\neq \boldsymbol{b}/g$, then for some
$1\le i\le k$ we have $a_i/h\neq b_i/g$. Hence
\[
\ord\!\left(\frac{a_i}{h}-\frac{b_i}{g}\right)\ge -\deg h-\deg g>-2\rho n
=-\frac{n}{2r^\ast}.
\]
On the other hand, by the ultrametric inequality and the definition of $N_n$,
\[
\ord\!\left(\frac{a_i}{h}-\frac{b_i}{g}\right)
<
-r_in+\frac{n}{4(r^\ast)^2}
\le -n+\frac{n}{4(r^\ast)^2}
<-\frac{n}{2r^\ast},
\]
which is a contradiction.  By assumption $h,g$ are monic $(a_1,...,a_k,h)=(b_1,...,b_k,g)=1$, so this forces $\boldsymbol{a}/h=\boldsymbol{b}/g$.

Now we state the minor arc estimate, which is a consequence of
Corollary \ref{epsilonfree43} with $\mathcal{K}=\{r_1,...,r_k\}$ and our chosen parameters.

\begin{lem}\label{minorarcestimate}
For every integer $n\ge 1$ and every $\boldsymbol{\alpha}\in \mathfrak{m}_n$, one has
\[
|\widehat{M}_n(\boldsymbol{\alpha})|\ll_{\mathcal{K},q} q^{-\delta_0 n}.
\]
\end{lem}
\begin{proof}
Write $\boldsymbol{\alpha}=(\alpha_1,\ldots,\alpha_k)$ and set $\eta:=\delta_0 n$.
Suppose for contradiction that
\[
|\widehat{M}_n(\boldsymbol{\alpha})|>q^{-\delta_0 n}=q^{-\eta}.
\]
Then
\[
\left|\sum_{\deg u<n}e(\alpha_1u^{r_1}+\cdots+\alpha_ku^{r_k})\right|
=q^n|\widehat{M}_n(\boldsymbol{\alpha})|>q^{n-\eta}.
\]
Since $\eta=\delta_0 n\le (c/2)n<cn$, we may apply Corollary \ref{epsilonfree43} with
$N=n$ and $\mathcal{K}=\{r_1,\ldots,r_k\}$. For this application of Corollary \ref{epsilonfree43} we must also assume $n$ is sufficiently large in terms of $\mathcal{K}$ and $q$, but the desired result can be guaranteed for these small $n$ by enlarging the implicit constant. 
We therefore obtain for
each $1\leq i\leq k$ polynomials $b_i\in\F_q[t]$ and monic $g_i\in\F_q[t]$ such that
\[
\ord(g_i\alpha_i-b_i)<-r_in+C\eta+D
\]
and
\[
\deg g_i\leq C\eta+D\le \frac{n}{16(r^\ast)^2}.
\]
Reducing $b_i/g_i$ to lowest terms, we may write $\frac{b_i}{g_i}=\frac{a_i'}{h_i}$
with $h_i$ monic and $(a_i',h_i)=1$. Then $\deg h_i\leq n/(8(r^\ast)^2)$ and
\[
\ord\!\left(\alpha_i-\frac{a_i'}{h_i}\right)
<-r_in-\deg h_i+C\eta+D
<-r_in+\frac{n}{4(r^\ast)^2}.
\]
Now let $h=\operatorname{lcm}(h_1,\ldots,h_k)$, and define $a_i:=a_i'h/h_i$. By construction,
$a_i/h=a_i'/h_i$, and if $d_i:=(a_i,h)$, then $h_i=h/d_i$. Moreover,
$(a_1,\ldots,a_k,h)=1$. Since the $r_i$ are distinct and $r_i\le r^\ast$, we have $k\le r^\ast$,
and hence
\[
\deg h\leq \sum_{i=1}^k\deg h_i
\leq k(C\delta_0n+D)
\leq \frac{n}{16r^\ast}
<\rho n.
\]
Moreover,
\[
\ord\!\left(\alpha_i-\frac{a_i}{h}\right)
=
\ord\!\left(\alpha_i-\frac{a_i'}{h_i}\right)
<-r_in+\frac{n}{4(r^\ast)^2}
\]
for every $1\leq i\leq k$. Therefore $\boldsymbol{\alpha}\in N_n(a_1,\ldots,a_k,h)\subset \mathfrak{M}_n$,
contradicting $\boldsymbol{\alpha}\in \mathfrak{m}_n$. This proves the lemma.
\end{proof}

Now we begin the proof of the major arc estimate. We will encounter the following type of Gauss sum. For $(a_1,...,a_k,h)=1$ define
\begin{equation}\label{LambdaHat}
\Lambda(a_1,...,a_k,h)
:=q^{-\deg h}\sum_{\deg f<\deg h}
e\!\left(\frac{a_1f^{r_1}+\cdots+a_kf^{r_k}}{h}\right).
\end{equation}
For later use, if $h$ is monic, define
\begin{equation}\label{AhNotation}
\mathcal{A}_h:=\bigl\{\boldsymbol{a}=(a_1,\ldots,a_k)\in \F_q[t]^k:
\deg a_i<\deg h\ \forall i,\ (a_1,...,a_k,h)=1\bigr\}.
\end{equation}
For $\boldsymbol{a}\in\mathcal{A}_h$, write
\begin{equation}\label{vectorFractionNotation}
\frac{\boldsymbol{a}}{h}:=\left(\frac{a_1}{h},\ldots,\frac{a_k}{h}\right),
\qquad
\Lambda(a_1,...,a_k,h):=\Lambda(a_1,\ldots,a_k,h).
\end{equation}
With these definitions in hand we derive the major arc estimate. This proof is inspired by Kubota's seminal work on Waring's problem over $\F_q[t]$ \cite{Kubota1974Waring}.

\begin{prop}[Major arc estimate]\label{prop:major-arc}
Assume that $n\in\mathbb{N}$. Let $h\in \F_q[t]$ be monic with $\deg h<\rho n$, let $a_1,\ldots,a_k\in \F_q[t]$ satisfy $\deg a_i<\deg h$ for every $1\le i\le k$, and assume that $(a_1,\ldots,a_k,h)=1$. If $(\alpha_1,\ldots,\alpha_k)\in N_n(a_1,\ldots,a_k,h)$ and
\[
\beta_i:=\alpha_i-\frac{a_i}{h}
\qquad (1\le i\le k),
\]
then
\begin{equation}\label{majorarcestimate}
\wh{M}_n(\alpha_1,\ldots,\alpha_k)
=
\Lambda(a_1,\ldots,a_k,h)\wh{M}_n(\beta_1,\ldots,\beta_k).
\end{equation}
\end{prop}

\begin{proof}
Since $(\alpha_1,\ldots,\alpha_k)\in N_n(a_1,\ldots,a_k,h)$, we have
\[
\ord(\beta_i)<-r_i n+\frac{n}{4(r^{\ast})^2}
\qquad (1\le i\le k).
\]
Every polynomial $u\in \F_q[t]$ with $\deg u<n$ can be written uniquely in the form $u=hy+z$,

where $\deg z<\deg h$ and $\deg y<n-\deg h$. Therefore
\begin{align*}
q^n\wh{M}_n(\alpha_1,\ldots,\alpha_k)
&=\sum_{\deg z<\deg h}\sum_{\deg y<n-\deg h}
 e\!\Big(\sum_{i=1}^k \alpha_i(hy+z)^{r_i}\Big).
\end{align*}
Fix such $y$ and $z$. For each $1\le i\le k$, the binomial theorem gives
\[
(hy+z)^{r_i}-(hy)^{r_i}=\sum_{j=1}^{r_i}\binom{r_i}{j}(hy)^{r_i-j}z^j.
\]
Since $\deg(hy)<n$ and $\deg z<\deg h$, each term on the right has degree at most
\[
(r_i-j)n+j\deg h\le (r_i-1)n+\deg h.
\]
It follows that
\[
\ord\!\Big(\beta_i\big((hy+z)^{r_i}-(hy)^{r_i}\big)\Big)
<-n+\deg h+\frac{n}{4(r^{\ast})^2}.
\]
Since $\deg h<\rho n=n/(8r^{\ast})$, the right-hand side is $<-1$ for all $n>1$. Therefore for $1\leq i\leq k$ we have 
$e\!\big(\beta_i(hy+z)^{r_i}\big)=e\!\big(\beta_i(hy)^{r_i}\big)$
and hence
\[
e\!\Big(\sum_{i=1}^k \beta_i(hy+z)^{r_i}\Big)=e\!\Big(\sum_{i=1}^k \beta_i(hy)^{r_i}\Big).
\]
Also, $(hy+z)^{r_i}\equiv z^{r_i}\pmod h$
for every $1\le i\le k$, so
\[
e\!\Big(\sum_{i=1}^k \frac{a_i}{h}(hy+z)^{r_i}\Big)=e\!\Big(\sum_{i=1}^k \frac{a_i}{h}z^{r_i}\Big).
\]
Combining these identities, we obtain
\begin{align*}
q^n\wh{M}_n(\alpha_1,\ldots,\alpha_k)
&=\Big(\sum_{\deg z<\deg h}e\!\Big(\frac{a_1z^{r_1}+\cdots+a_kz^{r_k}}{h}\Big)\Big)
\Big(\sum_{\deg y<n-\deg h}e\!\Big(\sum_{i=1}^k \beta_i(hy)^{r_i}\Big)\Big)\\
&=q^{\deg h}\Lambda(a_1,\ldots,a_k,h)
\sum_{\deg y<n-\deg h}e\!\Big(\sum_{i=1}^k \beta_i(hy)^{r_i}\Big).
\end{align*}
The same decomposition gives
\begin{align*}
q^n\wh{M}_n(\beta_1,\ldots,\beta_k)
&=\sum_{\deg z<\deg h}\sum_{\deg y<n-\deg h}e\!\Big(\sum_{i=1}^k \beta_i(hy+z)^{r_i}\Big)\\
&=q^{\deg h}\sum_{\deg y<n-\deg h}e\!\Big(\sum_{i=1}^k \beta_i(hy)^{r_i}\Big).
\end{align*}
Comparing the last two displays proves \eqref{majorarcestimate}.
\end{proof}

\section{Reduction to the major arc model}\label{majorarcmodel}

Recall the definitions of
$\mathcal{A}_h$, $\mathbf{a}/h$, and $N_n(a_1,\ldots,a_k,h)$ from the previous section.

Let
\[
\mathcal{B}_n:=\Bigl\{\boldsymbol{\beta}\in\mathbb{T}^k:\ord(\beta_i)<-r_in\ \forall i\Bigr\}.
\]
We define a multiplier $\wh{C}_n$ supported on the translated boxes
$\mathbf{a}/h+\mathcal{B}_n$, one for each rational point $\mathbf{a}/h$ with $\deg h<\rho n$, weighted
by $\Lambda(\mathbf{a},h)$. In particular, define
\[
\wh{C}_n(\boldsymbol{\alpha})
:=\sum_{\substack{h\ \mathrm{monic}\\ \deg h<\rho n}}
\ \sum_{\mathbf{a}\in\mathcal{A}_h}
\Lambda(\mathbf{a},h)
\1_{\mathcal{B}_n}\left(\boldsymbol{\alpha}-\frac{\mathbf{a}}{h}\right).
\] 
Notice that if $\boldsymbol{\beta}\in\mathcal{B}_n$, then $\wh{M}_n(\boldsymbol{\beta})=1$. Indeed, if $\deg f<n$, then
\[
\ord(\beta_i f^{r_i})<-r_in+r_i(n-1)<-1
\qquad (1\le i\le k),
\]
and hence
\[
e\!\Bigl(\beta_1f^{r_1}+\cdots+\beta_kf^{r_k}\Bigr)=1.
\]
For each pair $(\mathbf{a},h)$ define the frame
\[
\mathcal{F}_n(\mathbf{a},h):=N_n(\mathbf{a},h)\setminus\left(\frac{\mathbf{a}}{h}+\mathcal{B}_n\right).
\]
Then we may decompose $\mathbb{T}^k$ into major and minor arcs at scale $n$, and then we may decompose each major arc box into a translate of $\mathcal{B}_n$ and a frame for each $\boldsymbol{\alpha}\in \mathbb{T}^k$. With this decomposition in mind, using the disjointness of the major arcs and Proposition~\ref{prop:major-arc}, we immediately have 
\begin{equation}\label{decomposition}
\wh{M}_n(\boldsymbol{\alpha})-\wh{C}_n(\boldsymbol{\alpha})
=
\wh{M}_n(\boldsymbol{\alpha})\1_{\mathfrak{m}_n}(\boldsymbol{\alpha})
+
\sum_{\substack{h\ \mathrm{monic}\\ \deg h<\rho n}}
\sum_{\mathbf{a}\in \mathcal{A}_h}
\Lambda(\mathbf{a},h)
\wh{M}_n\left(\boldsymbol{\alpha}-\frac{\mathbf{a}}{h}\right)
\1_{\mathcal{F}_n(\mathbf{a},h)}(\boldsymbol{\alpha}).
\end{equation}

With this notation in hand, we are prepared to perform the reduction from $M_n$ to the major arc model $C_n$. 

\begin{prop}\label{outermajorreduction}
The sequences $M_n$ and $C_n$ are equivalent on $\ell^2(\F_q[t]^k)$.
\end{prop}

\begin{proof}
By \eqref{decomposition} and Plancherel's identity, we have 
\begin{align*}
\sum_{x\in\F_q[t]^k}\sum_{n=1}^\infty |M_ng(x)-C_ng(x)|^2=\sum_{n=1}^\infty\int_{\mathbb{T}^k}
\bigl|\wh{M}_n(\boldsymbol{\alpha})-\wh{C}_n(\boldsymbol{\alpha})\bigr|^2
|\wh{g}(\boldsymbol{\alpha})|^2\,dm(\boldsymbol{\alpha}).
\end{align*}
Notice that by Lemma \ref{minorarcestimate}, we have 
\[
 \sum_{n=1}^\infty\int_{\mathbb{T}^k}
\bigl|\wh{M}_n(\boldsymbol{\alpha})\1_{\mathfrak{m}_n}(\boldsymbol{\alpha})\bigr|^2|\wh{g}(\boldsymbol{\alpha})|^2\,dm(\boldsymbol{\alpha})
\leq \sum_{n=1}^\infty q^{-2\delta_0n}\|g\|_{\ell^2(\F_q[t]^k)}^2
\ll_{\mathcal{K},q}\|g\|_{\ell^2(\F_q[t]^k)}^2.
\]

Therefore, by also organizing the frames $\mathcal{F}_n(\mathbf{a},h)$ by the degree of $h$ and using the triangle inequality, it suffices to show 
\begin{equation}\label{framebound}
\sum_{n=1}^\infty\int_{\mathbb{T}^k}
\left|
\1_{s<\rho n}
\sum_{\substack{h\ \mathrm{monic}\\ \deg h=s}}
\sum_{\mathbf{a}\in \mathcal{A}_h}
\Lambda(\mathbf{a},h)
\wh{M}_n\left(\boldsymbol{\alpha}-\frac{\mathbf{a}}{h}\right)
\1_{\mathcal{F}_n(\mathbf{a},h)}(\boldsymbol{\alpha})
\right|^2
|\wh{g}(\boldsymbol{\alpha})|^2\,dm(\boldsymbol{\alpha})
\ll_{\mathcal{K},q}q^{-2\gamma s}\|g\|_{\ell^2(\F_q[t]^k)}^2
\end{equation}
for every fixed $s\ge 0$.

For ease of notation define the multiplier
\[
\wh{F}_{s,n}(\boldsymbol{\alpha})
:=
\1_{s<\rho n}
\sum_{\substack{h\ \mathrm{monic}\\ \deg h=s}}
\sum_{\mathbf{a}\in \mathcal{A}_h}
\Lambda(\mathbf{a},h)
\wh{M}_n\left(\boldsymbol{\alpha}-\frac{\mathbf{a}}{h}\right)
\1_{\mathcal{F}_n(\mathbf{a},h)}(\boldsymbol{\alpha}).
\]
Fix $s\ge 0$ and $\boldsymbol{\alpha}\in\mathbb{T}^k$. 

Notice that if $h_1,h_2$ are monic with 
\[
\deg h_1=\deg h_2=s<\rho n
\]
and $(\mathbf{a}_1,h_1)=(\mathbf{a}_2,h_2)=1$ with 
\[
\boldsymbol{\alpha}\in N_n(\mathbf{a}_1,h_1)\cap N_n(\mathbf{a}_2,h_2),
\]
then the ultrametric inequality easily gives $\mathbf{a}_1/h_1=\mathbf{a}_2/h_2$. 
Thus, for a fixed pair $(s,\boldsymbol{\alpha})$ there is at most one
reduced rational $\mathbf{a}/h$ with $\deg h=s$ and $h$-monic such that there exists an integer $n_0$ satisfying
$s<\rho n_0$ and $\boldsymbol{\alpha}\in N_{n_0}(\mathbf{a},h)$.
If no such rational $\mathbf{a}/h$ exists, then $\wh{F}_{s,n}(\boldsymbol{\alpha})= 0$ for all $n$, and we are done. 
Otherwise, set $\boldsymbol{\beta}:=\boldsymbol{\alpha}-\frac{\mathbf{a}}{h}.$
Whenever $\wh{F}_{s,n}(\boldsymbol{\alpha})\neq 0$, we have
$\boldsymbol{\alpha}\in \mathcal{F}_n(\mathbf{a},h)$, and therefore
\[
\Delta_n(\boldsymbol{\beta})
:=\max_{1\le i\le k}\bigl(\ord(\beta_i)+r_in\bigr)
\ge 0,
\]
because $\boldsymbol{\beta}\notin\mathcal{B}_n$. On the other hand,
$\boldsymbol{\alpha}\in N_n(\mathbf{a},h)$ implies, for $1\leq i\leq k$,
\[
\ord(\beta_i)<-r_in+\frac{n}{4(r^\ast)^2},
\]
and hence
\[
0\le \Delta_n(\boldsymbol{\beta})<\frac{n}{4(r^\ast)^2}<\frac n2.
\]

Applying Corollary \ref{smallbeta_inverse_ultimate} to an index at which the maximum defining
$\Delta_n(\boldsymbol{\beta})$ is attained, after relabeling the exponents if necessary, we obtain a constant $c>0$ depending on $\mathcal{K}$ and $q$ such that for all $n\in \N$,
\[
\bigl|\wh{M}_n(\boldsymbol{\beta})\bigr|
\ll q^{-c\Delta_n(\boldsymbol{\beta})}.
\]
Here we also used that $(r_i,p)=1$ for all $1\leq i\leq k$ so that Corollary \ref{smallbeta_inverse_ultimate} applies to each exponent in $\mathcal{K}$. Combining this with the Gauss sum bound from Proposition~\ref{GaussSumBound} yields
\[
|\wh{F}_{s,n}(\boldsymbol{\alpha})|^2
\ll q^{-2\gamma s}\,q^{-2c\Delta_n(\boldsymbol{\beta})}
\]
whenever $\wh{F}_{s,n}(\boldsymbol{\alpha})\neq 0$. Therefore it is enough to show that
\[
\sum_{\substack{n\ge 1\\ \wh{F}_{s,n}(\boldsymbol{\alpha})\neq 0}}
q^{-2c\Delta_n(\boldsymbol{\beta})}
\ll_{\mathcal{K},q} 1.
\]

Observe that the quantity $\Delta_n(\boldsymbol{\beta})$ increases by at least $1$ at each step.
Indeed, if $i_0$ is an index at which the maximum defining $\Delta_n(\boldsymbol{\beta})$ is attained, then
\begin{align*}
\Delta_{n+1}(\boldsymbol{\beta})
&=\max_{1\le i\le k}\bigl(\ord(\beta_i)+r_i(n+1)\bigr)\\
&\ge \ord(\beta_{i_0})+r_{i_0}(n+1)\\
&\ge \Delta_n(\boldsymbol{\beta})+1.
\end{align*}
Since $\Delta_n(\boldsymbol{\beta})\ge 0$ whenever $\wh{F}_{s,n}(\boldsymbol{\alpha})\neq 0$, it follows that
the nonnegative values $\Delta_n(\boldsymbol{\beta})$ arising from the set
$\{n\ge 1:\wh{F}_{s,n}(\boldsymbol{\alpha})\neq 0\}$ are separated by at least $1$. Therefore
\[
\sum_{\substack{n\ge 1\\ \wh{F}_{s,n}(\boldsymbol{\alpha})\neq 0}} q^{-2c\Delta_n(\boldsymbol{\beta})}
\le \sum_{m\ge 0} q^{-2cm}
\ll_{\mathcal{K},q} 1.
\]
This proves $\displaystyle
\sum_{n=1}^\infty |\wh{F}_{s,n}(\boldsymbol{\alpha})|^2\ll_{\mathcal{K},q} q^{-2\gamma s}.
$
Consequently,
\[
\sum_{n=1}^\infty\int_{\mathbb{T}^k}|\wh{F}_{s,n}(\boldsymbol{\alpha})|^2|\wh g(\boldsymbol{\alpha})|^2\,dm(\boldsymbol{\alpha})
\ll_{\mathcal{K},q} q^{-2\gamma s}\|g\|_{\ell^2(\F_q[t]^k)}^2,
\]
which proves \eqref{framebound}. Finally, by the triangle inequality in $\ell^2(\mathbb{N}\times\mathbb{T}^k)$,
\[
\left(\sum_{n=1}^\infty\int_{\mathbb{T}^k}\left|\sum_{s\ge0}\wh{F}_{s,n}(\boldsymbol{\alpha})\right|^2|\wh g(\boldsymbol{\alpha})|^2\,dm(\boldsymbol{\alpha})\right)^{1/2}
\ll_{\mathcal{K},q}\sum_{s\ge0}q^{-\gamma s}\|g\|_{\ell^2(\F_q[t]^k)}
\ll_{\mathcal{K},q}\|g\|_{\ell^2(\F_q[t]^k)}.
\]
Combining this with the minor arc estimate proves the proposition.
\end{proof}

\section{A monotone sequence of projections}

Recall the definition of the major arc model $C_n$ from the previous section. For each
nonnegative integer $s$, define the degree $s$ piece $C_{s,n}$ by
\[
\widehat{C}_{s,n}(\boldsymbol{\alpha})
:=
\1_{s<\rho n}
\sum_{\substack{h\ \mathrm{monic}\\ \deg h=s}}
\sum_{\boldsymbol{a}\in\mathcal{A}_h}
\Lambda(\boldsymbol{a},h)
\1_{\mathcal{B}_n}\left(\boldsymbol{\alpha}-\frac{\boldsymbol{a}}{h}\right).
\]
Thus
\[
C_n=\sum_{s\ge 0}C_{s,n}.
\]
By summing over $s$, we reduce the proof of the oscillation inequality for $C_n$ to the following.

\begin{prop}\label{Crosc}
For every finite increasing sequence $n_1<\cdots<n_{t_0}$ of positive integers,
\begin{equation}
\|\mathcal{O}_{n_1,\ldots,n_{t_0}}(C_{s,n}g)\|_{\ell^2(\F_q[t]^k)}
\ll q^{-\gamma s}\|g\|_{\ell^2(\F_q[t]^k)}.
\end{equation}
\end{prop}

The point of the current section is that $C_{s,n}$ can be realized in terms of a monotone sequence
of projections. As in Quas--Wierdl, the oscillation inequality for a monotone sequence of
projections is tractable once one obtains an appropriate maximal inequality.

For each nonnegative integer $s$, define
\[
\mathcal{E}_s:=\left\{\frac{\boldsymbol{a}}{h}: h\ \text{monic},\ \deg h=s,\ \boldsymbol{a}\in\mathcal{A}_h\right\}\subset\mathbb{T}^k.
\]
Also define
\[
N_s:=\min\{n\in\N:s<\rho n\}.
\]
For $n\ge 1$, define a Fourier multiplier operator $D_{s,n}:\ell^2(\F_q[t]^k)\to \ell^2(\F_q[t]^k)$ by
\[
\widehat D_{s,n}(\boldsymbol{\alpha})
:=
\1_{s<\rho n}
\sum_{\substack{h\ \mathrm{monic}\\ \deg h=s}}
\sum_{\boldsymbol{a}\in\mathcal{A}_h}
\1_{\mathcal{B}_n}\left(\boldsymbol{\alpha}-\frac{\boldsymbol{a}}{h}\right).
\]
Observe that $D_{s,n}$ satisfies the following key property as a consequence of our choice of
$\rho$ and the disjointness of the major arcs at scale $n$ when $s<\rho n$:
\begin{equation}\label{disjoint}
\widehat D_{s,n}(\boldsymbol{\alpha})\in\{0,1\}.
\end{equation}
Thus $D_{s,n}$ is a projection. For $n\ge N_s$, the sets
\[
U_{s,n}:=\operatorname{supp}(\widehat D_{s,n})
\]
are decreasing in $n$, because the boxes $\mathcal{B}_n$ shrink as $n$ increases while their
centers remain fixed. We will call a sequence of projections $(P_n)$ monotone if
\[
P_mP_n=P_nP_m=P_n
\qquad\text{whenever }m\le n.
\]
Therefore $(D_{s,n})_{n\ge N_s}$ is a monotone sequence of projections.

To prove Proposition \ref{Crosc}, it suffices to show the oscillation inequality for $D_{s,n}$ with at most
a logarithmic loss in
\[
R_s:=\deg Q_s,
\qquad
Q_s:=\prod_{\substack{h\ \mathrm{monic}\\ \deg h=s}}h.
\]

\begin{prop}\label{Drosc}
For every finite increasing sequence $n_1<\cdots<n_{t_0}$ of positive integers,
\begin{equation}
\|\mathcal{O}_{n_1,\ldots,n_{t_0}}(D_{s,n}g)\|_{\ell^2(\F_q[t]^k)}
\ll \log(2+R_s)\|g\|_{\ell^2(\F_q[t]^k)}.
\end{equation}
\end{prop}

\begin{proof}[Proof of Proposition \ref{Crosc} assuming Proposition \ref{Drosc}]
Define $w_s:\F_q[t]^k\to \mathbb{C}$ by
\[
\widehat{w_s}(\boldsymbol{\alpha})
:=
\sum_{\substack{h\ \mathrm{monic}\\ \deg h=s}}
\sum_{\boldsymbol{a}\in\mathcal{A}_h}
\Lambda(\boldsymbol{a},h)
\1_{\mathcal{B}_{N_s}}\left(\boldsymbol{\alpha}-\frac{\boldsymbol{a}}{h}\right)
\widehat{g}(\boldsymbol{\alpha}).
\]
Since the boxes in the definition of $\widehat{w_s}$ are pairwise disjoint, the Gauss sum bound from
Proposition~\ref{GaussSumBound} implies
\[
\|w_s\|_{\ell^2(\F_q[t]^k)}^2\ll q^{-2\gamma s}\|g\|_{\ell^2(\F_q[t]^k)}^2.
\]
Moreover, for every $n\ge N_s$ we have
\[
D_{s,n}w_s=C_{s,n}g.
\]
For $n<N_s$, both sides are zero. Therefore Proposition \ref{Drosc} yields
\[
\|\mathcal{O}_{n_1,\ldots,n_{t_0}}(C_{s,n}g)\|_{\ell^2(\F_q[t]^k)}
=
\|\mathcal{O}_{n_1,\ldots,n_{t_0}}(D_{s,n}w_s)\|_{\ell^2(\F_q[t]^k)}
\ll \log(2+R_s)\|w_s\|_{\ell^2(\F_q[t]^k)}.
\]
Thus
\[
\|\mathcal{O}_{n_1,\ldots,n_{t_0}}(C_{s,n}g)\|_{\ell^2(\F_q[t]^k)}
\ll \log(2+R_s)q^{-\gamma s}\|g\|_{\ell^2(\F_q[t]^k)}.
\]
Since there are exactly $q^s$ monic polynomials of degree $s$, we have
\[
R_s=sq^s.
\]
Hence, after possibly decreasing $\gamma$, the logarithmic factor may be absorbed into the
exponential decay, and we recover Proposition \ref{Crosc}.
\end{proof}

Because $D_{s,n}$ is a monotone sequence of projections for $n\ge N_s$, to prove Proposition
\ref{Drosc} it suffices to show that $D_{s,n}$ satisfies a maximal inequality. Define the maximal
operator $D_s^\ast$ by
\[
D_s^{\ast}g(\boldsymbol{x})
:=
\sup_{n\ge N_s}|D_{s,n}g(\boldsymbol{x})|,
\]
for $\boldsymbol{x}\in\F_q[t]^k$.

\begin{prop}\label{Dmax}
\begin{equation}
\|D_s^{\ast}g\|_{\ell^2(\F_q[t]^k)}
\ll \log(2+R_s)\|g\|_{\ell^2(\F_q[t]^k)}.
\end{equation}
\end{prop}

\begin{proof}[Proof of Proposition \ref{Drosc} assuming Proposition \ref{Dmax}]
Since $D_{s,n}$ is a monotone sequence of projections for $n\ge N_s$, if
$N_s\le m_1\le m_2$, then
\[
D_{s,m_1}D_{s,m_2}=D_{s,m_2}D_{s,m_1}=D_{s,m_2}.
\]
Therefore, if $N_s\le n_j\le n\le n_{j+1}$, then
\[
D_{s,n}g-D_{s,n_{j+1}}g
=
D_{s,n}\bigl(D_{s,n_j}g-D_{s,n_{j+1}}g\bigr).
\]
The intervals in the oscillation functional which lie below $N_s$ contribute nothing, and there is
at most one interval which crosses the threshold $N_s$. The contribution of this crossing interval
is bounded by $2D_s^\ast g$. Thus, by Proposition \ref{Dmax},
\begin{align*}
\|\mathcal{O}_{n_1,\ldots,n_{t_0}}(D_{s,n}g)\|_{\ell^2(\F_q[t]^k)}^2
&\ll
(\log(2+R_s))^2\|g\|_{\ell^2(\F_q[t]^k)}^2\\
&\quad+
\sum_j\sum_{\boldsymbol{x}\in \F_q[t]^k}
\sup_{n_j\le n<n_{j+1}}
\bigl|D_{s,n}(D_{s,n_j}g-D_{s,n_{j+1}}g)(\boldsymbol{x})\bigr|^2,
\end{align*}
where the sum is restricted to those $j$ with $n_j\ge N_s$. Applying Proposition \ref{Dmax} again, we obtain
\begin{align*}
\|\mathcal{O}_{n_1,\ldots,n_{t_0}}(D_{s,n}g)\|_{\ell^2(\F_q[t]^k)}^2
&\ll
(\log(2+R_s))^2
\left(
\|g\|_{\ell^2(\F_q[t]^k)}^2
+
\sum_j
\|D_{s,n_j}g-D_{s,n_{j+1}}g\|_{\ell^2(\F_q[t]^k)}^2
\right).
\end{align*}
Finally, by Plancherel,
\begin{align*}
\sum_j
\|D_{s,n_j}g-D_{s,n_{j+1}}g\|_{\ell^2(\F_q[t]^k)}^2
&=
\int_{\mathbb{T}^k}|\widehat g(\boldsymbol{\alpha})|^2
\sum_j
\big|\widehat D_{s,n_j}(\boldsymbol{\alpha})
-\widehat D_{s,n_{j+1}}(\boldsymbol{\alpha})\big|^2
\,dm(\boldsymbol{\alpha}).
\end{align*}
It remains to show
\begin{equation}\label{disjointsum}
\sum_j
\big|\widehat D_{s,n_j}(\boldsymbol{\alpha})
-\widehat D_{s,n_{j+1}}(\boldsymbol{\alpha})\big|^2
\le 1
\end{equation}
for all $\boldsymbol{\alpha}\in\mathbb T^k$, where the sum is over those $j$ with $n_j\ge N_s$.
Set
\[
U_j:=\operatorname{supp}(\widehat D_{s,n_j}).
\]
Since $(D_{s,n})_{n\ge N_s}$ is a monotone sequence of projections, we have
$\widehat D_{s,n_j}(\boldsymbol{\alpha})\in\{0,1\}$ and the sets $U_j$ are decreasing. Hence
\[
\big|\widehat D_{s,n_j}(\boldsymbol{\alpha})
-\widehat D_{s,n_{j+1}}(\boldsymbol{\alpha})\big|^2
=
\1_{U_j\setminus U_{j+1}}(\boldsymbol{\alpha}),
\]
and the sets $U_j\setminus U_{j+1}$ are pairwise disjoint as $j$ varies. Therefore
\[
\sum_j
\big|\widehat D_{s,n_j}(\boldsymbol{\alpha})
-\widehat D_{s,n_{j+1}}(\boldsymbol{\alpha})\big|^2
\le 1.
\]
This proves \eqref{disjointsum} and hence Proposition \ref{Drosc}.
\end{proof}

\section{The small scales}
We have reduced the problem to showing Proposition \ref{Dmax}, the maximal inequality for $D_{s,n}$. The idea is to split $n\ge N_s$ into two cases based on the size of $n$, which we call the \textit{small scales} and \textit{large scales}. We treat the small scales first. This is a standard Rademacher-Menshov type argument, but we include it for the sake of completeness. 

Recall from the previous section that
\[
R_s=\deg Q_s,
\qquad
Q_s=\prod_{\substack{h\ \mathrm{monic}\\ \deg h=s}}h.
\]
The small scales are the integers $n$ with
\[
N_s\le n<R_s^4.
\]
If $N_s\ge R_s^4$, there is nothing to prove in this range, so assume from now on that
$N_s<R_s^4$. Set
\[
T:=R_s^4
\qquad\text{and}\qquad
L:=T-N_s.
\]
For convenience, define
\[
P_0:=0,
\qquad
P_m:=D_{s,T-m}\quad (1\le m\le L),
\]
and set
\[
P_m:=D_{s,N_s}\quad (m>L).
\]
Then $(P_m)$ is an increasing sequence of projections. Moreover,
\[
\sup_{N_s\le n<R_s^4}|D_{s,n}g(\boldsymbol{x})|
=
\sup_{1\le m\le L}|P_mg(\boldsymbol{x})|
\le
\sup_{0\le m\le 2^M}|P_mg(\boldsymbol{x})|,
\]
where
\[
M:=\left\lceil \log_2 L\right\rceil.
\]

For each $0\le n\le 2^M$, by considering the binary expansion of $n$, we may write $P_n$ in the form
\[
P_n=\sum_{\ell=0}^M
\bigl(P_{(d_\ell+1)2^\ell}-P_{d_\ell2^\ell}\bigr),
\]
where $d_\ell\in\{0,1,\ldots,2^{M-\ell}-1\}$ for each $\ell$, with the convention that some of these terms may be absent. Consequently, by Cauchy's inequality,
\[
\sup_{0\le n\le 2^M}|P_ng(\boldsymbol{x})|^2
\le
(M+1)
\sum_{\ell=0}^M
\sum_{0\le d<2^{M-\ell}}
\big|P_{(d+1)2^\ell}g(\boldsymbol{x})-P_{d2^\ell}g(\boldsymbol{x})\big|^2.
\]
By the same argument used to prove \eqref{disjointsum}, for each $\boldsymbol{\alpha}\in\mathbb{T}^k$ one has
\[
\sum_{d=0}^{2^{M-\ell}-1}
\big|\widehat P_{(d+1)2^\ell}(\boldsymbol{\alpha})-\widehat P_{d2^\ell}(\boldsymbol{\alpha})\big|^2\le 1.
\]
Combining these estimates and applying Plancherel yields
\begin{align*}
\sum_{\boldsymbol{x}\in \F_q[t]^k}\sup_{N_s\le n<R_s^4}|D_{s,n}g(\boldsymbol{x})|^2
&\le
(M+1)
\sum_{\ell=0}^M
\sum_{0\le d<2^{M-\ell}}
\sum_{\boldsymbol{x}\in \F_q[t]^k}
\big|P_{(d+1)2^\ell}g(\boldsymbol{x})-P_{d2^\ell}g(\boldsymbol{x})\big|^2\\
&=
(M+1)
\sum_{\ell=0}^M
\int_{\mathbb{T}^k}|\widehat g(\boldsymbol{\alpha})|^2
\sum_{0\le d<2^{M-\ell}}
\big|\widehat P_{(d+1)2^\ell}(\boldsymbol{\alpha})-\widehat P_{d2^\ell}(\boldsymbol{\alpha})\big|^2
\,dm(\boldsymbol{\alpha})\\
&\le
(M+1)^2
\int_{\mathbb{T}^k}|\widehat g(\boldsymbol{\alpha})|^2\,dm(\boldsymbol{\alpha})\\
&=
(M+1)^2
\sum_{\boldsymbol{x}\in \F_q[t]^k}|g(\boldsymbol{x})|^2.
\end{align*}
Since $L\le R_s^4$, we have
\[
M+1\ll \log(2+R_s).
\]
Thus the contribution of the small scales is bounded by a logarithmic factor, which is acceptable.

\section{The large scales}

We now treat the \emph{large scales}, namely the range
\[
n\ge \max\{N_s,R_s^4\}.
\]
At this point the argument from Appendix~B of Bergelson's survey simplifies slightly in the function-field setting: because the geometric sum identity on $\F_q[t]$ is exact, the auxiliary linear operators that appear in \cite[Appendix~B]{Bergelson2006}
are already built into the projections $D_{s,n}$.

For ease of reference, we recall that $D_{s,n}:\ell^2(\F_q[t]^k)\to \ell^2(\F_q[t]^k)$ is given by
\[
\widehat D_{s,n}(\boldsymbol{\alpha})
=
\1_{s<\rho n}
\sum_{\substack{h\ \mathrm{monic}\\ \deg h=s}}
\sum_{\boldsymbol{a}\in\mathcal{A}_h}
\1_{\mathcal{B}_n}\left(\boldsymbol{\alpha}-\frac{\boldsymbol{a}}{h}\right),
\]
where
\[
\mathcal{B}_n
=
\{(\beta_1,\ldots,\beta_k)\in\mathbb{T}^k:
\ord \beta_i<-r_in\ \forall\,1\leq i\leq k\},
\]
and
\[
Q_s=\prod_{\substack{h\ \mathrm{monic}\\ \deg h=s}}h
\qquad\text{and}\qquad
R_s=\deg Q_s.
\]
For $\boldsymbol{\beta}=(\beta_1,\ldots,\beta_k)\in\mathbb{T}^k$, define
\[
V_{R_s}(\boldsymbol{\beta})
:=
\prod_{i=1}^k\1_{\ord(\beta_i)<-R_s}.
\]
By the geometric sum identity in each coordinate,
\[
V_{R_s}(\boldsymbol{\beta})
=
q^{-kR_s}
\sum_{\substack{\boldsymbol{b}=(b_1,\ldots,b_k)\in\F_q[t]^k\\ \deg b_i<R_s\ \forall i}}
e(b_1\beta_1+\cdots+b_k\beta_k),
\]
and similarly
\[
\1_{\mathcal{B}_n}(\boldsymbol{\beta})
=
q^{-(r_1+\cdots+r_k)n}
\sum_{\substack{\boldsymbol{f}=(f_1,\ldots,f_k)\in\F_q[t]^k\\ \deg f_i<r_in\ \forall i}}
e(f_1\beta_1+\cdots+f_k\beta_k).
\]

Now, for $n\ge \max\{N_s,R_s^4\}$, the the geometric sum identity yields
\begin{align*}
\widehat{D_{s,n}g}(\boldsymbol{\alpha})
&=
\sum_{\substack{h\ \mathrm{monic}\\ \deg h=s}}
\sum_{\boldsymbol{a}\in\mathcal{A}_h}
\1_{\mathcal{B}_n}\left(\boldsymbol{\alpha}-\frac{\boldsymbol{a}}{h}\right)\widehat{g}(\boldsymbol{\alpha})\\
&=
q^{-(r_1+\cdots+r_k)n}
\sum_{\substack{h\ \mathrm{monic}\\ \deg h=s}}
\sum_{\boldsymbol{a}\in\mathcal{A}_h}
\sum_{\substack{\boldsymbol{f}=(f_1,\ldots,f_k)\in\F_q[t]^k\\ \deg f_i<r_in\ \forall i}}
e\left(\boldsymbol{f}\cdot \left(\boldsymbol{\alpha}-\frac{\boldsymbol{a}}{h}\right) \right)\widehat{g}(\boldsymbol{\alpha})\\
&=
q^{-(r_1+\cdots+r_k)n}
\sum_{\substack{h\ \mathrm{monic}\\ \deg h=s}}
\sum_{\boldsymbol{a}\in\mathcal{A}_h}
\sum_{\substack{\boldsymbol{f}=(f_1,\ldots,f_k)\in\F_q[t]^k\\ \deg f_i<r_in\ \forall i}}
e\left(-\boldsymbol{f}\cdot\frac{\boldsymbol{a}}{h}\right)
e\left(\boldsymbol{f}\cdot \boldsymbol{\alpha}\right)\widehat{g}(\boldsymbol{\alpha}).
\end{align*}
Thus, by Fourier inversion term by term, we have
\begin{equation}\label{Drnphysical}
D_{s,n}g(\boldsymbol{x})
=
q^{-(r_1+\cdots+r_k)n}
\sum_{\substack{h\ \mathrm{monic}\\ \deg h=s}}
\sum_{\boldsymbol{a}\in\mathcal{A}_h}
\sum_{\substack{\boldsymbol{f}\in\F_q[t]^k\\ \deg f_i<r_in\ \forall i}}
e\left(-\boldsymbol{f}\cdot\frac{\boldsymbol{a}}{h}\right)
g(\boldsymbol{x}+\boldsymbol{f}).
\end{equation}
We may write each coordinate uniquely as
\[
f_i=Q_su_i+b_i,
\qquad
\deg b_i<R_s,
\qquad
\deg u_i<r_in-R_s,
\]
because $n\ge R_s^4$ implies $r_in\ge R_s$ for every $i$. Moreover,
\[
e\left(-Q_su_i\frac{a_i}{h}\right)=1
\]
for every monic $h$ with $\deg h=s$, since $h\mid Q_s$. Substituting this decomposition into \eqref{Drnphysical}, we obtain the exact identity
\[
D_{s,n}g(\boldsymbol{x})=\mathcal{L}_{s,n}G_s(\boldsymbol{x}),
\]
where
\[
\mathcal{L}_{s,n}F(\boldsymbol{x})
:=
q^{-\sum_{i=1}^k(r_in-R_s)}
\sum_{\substack{\boldsymbol{u}=(u_1,\ldots,u_k)\in\F_q[t]^k\\ \deg u_i<r_in-R_s\ \forall i}}
F(\boldsymbol{x}+Q_s\boldsymbol{u}),
\]
and
\[
G_s(\boldsymbol{x})
:=
q^{-kR_s}
\sum_{\substack{\boldsymbol{b}=(b_1,\ldots,b_k)\in\F_q[t]^k\\ \deg b_i<R_s\ \forall i}}
\sum_{\substack{h\ \mathrm{monic}\\ \deg h=s}}
\sum_{\boldsymbol{a}\in\mathcal{A}_h}
e\left(-\boldsymbol{b}\cdot\frac{\boldsymbol{a}}{h}\right)
g(\boldsymbol{x}+\boldsymbol{b}).
\]
Applying the linear maximal inequality proved Appendix \ref{linearmax} to the operators $\mathcal{L}_{s,n}$ gives
\begin{equation}\label{largescalemax}
\left\|\sup_{n\ge \max\{N_s,R_s^4\}}|D_{s,n}g|\right\|_{\ell^2(\F_q[t]^k)}
\ll
\|G_s\|_{\ell^2(\F_q[t]^k)}.
\end{equation}
It remains to estimate $G_s$. By construction and the geometric sum identity,
\[
\widehat{G_s}(\boldsymbol{\alpha})
=
\sum_{\substack{h\ \mathrm{monic}\\ \deg h=s}}
\sum_{\boldsymbol{a}\in\mathcal{A}_h}
V_{R_s}\left(\boldsymbol{\alpha}-\frac{\boldsymbol{a}}{h}\right)\widehat{g}(\boldsymbol{\alpha}).
\]
There are exactly $q^s$ monic polynomials of degree $s$, so
\[
R_s=sq^s\ge 2s.
\]
Hence the same separation argument used to prove \eqref{disjoint} shows that the boxes
\[
\left\{\boldsymbol{\alpha}\in\mathbb{T}^k:
V_{R_s}\left(\boldsymbol{\alpha}-\frac{\boldsymbol{a}}{h}\right)=1
\right\},
\qquad
h\ \mathrm{monic},\ \deg h=s,\ \boldsymbol{a}\in\mathcal{A}_h,
\]
are pairwise disjoint. Therefore
\[
\sum_{\substack{h\ \mathrm{monic}\\ \deg h=s}}
\sum_{\boldsymbol{a}\in\mathcal{A}_h}
V_{R_s}\left(\boldsymbol{\alpha}-\frac{\boldsymbol{a}}{h}\right)
\le 1
\qquad \text{for every }\boldsymbol{\alpha}\in\mathbb{T}^k.
\]
Plancherel gives
\[
\|G_s\|_{\ell^2(\F_q[t]^k)}
\le
\|g\|_{\ell^2(\F_q[t]^k)}.
\]
Combining this with \eqref{largescalemax}, we obtain
\[
\left\|\sup_{n\ge \max\{N_s,R_s^4\}}|D_{s,n}g|\right\|_{\ell^2(\F_q[t]^k)}
\ll
\|g\|_{\ell^2(\F_q[t]^k)}.
\]
Together with the small-scale estimate proved in the previous section, this yields Proposition \ref{Dmax}.

\appendix 
\section{A Hardy-Littlewood type maximal inequality}\label{linearmax}
The goal of this section is to prove a standard Hardy--Littlewood type maximal inequality in the function field setting, but we record the proof for the sake of completeness.

Recall the definition of $\mathcal{L}_{s,n}$ for
\[
n\ge H_s:=\max\{N_s,R_s^4\},
\]
where
\[
Q_s=\prod_{\substack{h\ \mathrm{monic}\\ \deg h=s}}h
\qquad\text{and}\qquad
R_s=\deg Q_s.
\]
We have
\[
\mathcal{L}_{s,n}F(\boldsymbol{x})
=
q^{-\sum_{i=1}^k(r_in-R_s)}
\sum_{\substack{\boldsymbol{u}=(u_1,\ldots,u_k)\in\F_q[t]^k\\ \deg u_i<r_in-R_s\ \forall i}}
F(\boldsymbol{x}+Q_s\boldsymbol{u}).
\]
Define the maximal operator $\mathcal{L}_s^\ast$ by
\[
\mathcal{L}_s^\ast F(\boldsymbol{x})
:=
\sup_{n\ge H_s}\mathcal{L}_{s,n}|F|(\boldsymbol{x}).
\]
Since
\[
|\mathcal{L}_{s,n}F|\le \mathcal{L}_{s,n}|F|,
\]
the estimate \eqref{linearmax} follows from
\[
\left\|\mathcal{L}^\ast_sF\right\|_{\ell^2(\F_q[t]^k)}
\ll \|F\|_{\ell^2(\F_q[t]^k)}.
\]
The operator $\mathcal{L}_s^\ast$ is strong type $(\infty,\infty)$ with norm at most $1$. Thus, by Marcinkiewicz interpolation, it suffices to prove that $\mathcal{L}_s^\ast$ is weak type $(1,1)$. This follows from a Vitali type covering argument.

\begin{prop}
Let $\mathcal{L}_s^\ast$ be defined as above. Let $F\in \ell^1(\F_q[t]^k)$. Then, for every $\alpha>0$,
\[
\mu\bigl(\{\boldsymbol{x}\in \F_q[t]^k: \mathcal{L}^\ast_sF(\boldsymbol{x})>\alpha\}\bigr)
\leq
\frac{1}{\alpha}
\sum_{\boldsymbol{x}\in \F_q[t]^k}|F(\boldsymbol{x})|,
\]
where $\mu$ is the counting measure on $\F_q[t]^k$.
\end{prop}

\begin{proof}
Let $\alpha>0$ and define
\[
E_\alpha:=\{\boldsymbol{x}\in \F_q[t]^k: \mathcal{L}^\ast_sF(\boldsymbol{x})>\alpha\}.
\]
For each $\boldsymbol{x}\in E_\alpha$, choose $n_{\boldsymbol{x}}\ge H_s$ such that
\[
\mathcal{L}_{s,n_{\boldsymbol{x}}}|F|(\boldsymbol{x})>\alpha.
\]
For $n\ge H_s$, define
\[
B_{Q_s,n}
:=
\{Q_s\boldsymbol{u}:\boldsymbol{u}\in\F_q[t]^k,\ \deg u_i<r_in-R_s\ \forall 1\le i\le k\}.
\]
Then $B_{Q_s,n}$ is an additive subgroup of $\F_q[t]^k$, and
\[
|B_{Q_s,n}|=q^{\sum_{i=1}^k(r_in-R_s)}.
\]
The condition $\mathcal{L}_{s,n_{\boldsymbol{x}}}|F|(\boldsymbol{x})>\alpha$ is therefore equivalent to
\[
\sum_{\boldsymbol{u}\in B_{Q_s,n_{\boldsymbol{x}}}}
|F(\boldsymbol{x}+\boldsymbol{u})|
>
\alpha |B_{Q_s,n_{\boldsymbol{x}}}|.
\]

We now use the nesting of these sets. If
\[
\boldsymbol{x}_1+B_{Q_s,n_{\boldsymbol{x}_1}}
\cap
\boldsymbol{x}_2+B_{Q_s,n_{\boldsymbol{x}_2}}
\neq \emptyset
\]
and $n_{\boldsymbol{x}_1}\ge n_{\boldsymbol{x}_2}$, then
\[
B_{Q_s,n_{\boldsymbol{x}_2}}\subseteq B_{Q_s,n_{\boldsymbol{x}_1}}.
\]
Moreover, the intersection assumption gives
\[
\boldsymbol{x}_2-\boldsymbol{x}_1\in B_{Q_s,n_{\boldsymbol{x}_1}}.
\]
Hence
\[
\boldsymbol{x}_2+B_{Q_s,n_{\boldsymbol{x}_2}}
\subseteq
\boldsymbol{x}_1+B_{Q_s,n_{\boldsymbol{x}_1}}.
\]

Now let $T\subseteq E_\alpha$ be finite. By selecting greedily a point with maximal
$n_{\boldsymbol{x}}$ among the points not yet covered, we may find finitely many
$\boldsymbol{x}_1,\ldots,\boldsymbol{x}_m\in T$ such that the sets
\[
\boldsymbol{x}_j+B_{Q_s,n_{\boldsymbol{x}_j}}
\qquad (1\le j\le m)
\]
are pairwise disjoint and cover $T$. Therefore
\begin{align*}
|T|
&\le
\sum_{j=1}^m |B_{Q_s,n_{\boldsymbol{x}_j}}|\\
&<
\frac{1}{\alpha}
\sum_{j=1}^m
\sum_{\boldsymbol{u}\in B_{Q_s,n_{\boldsymbol{x}_j}}}
|F(\boldsymbol{x}_j+\boldsymbol{u})|\\
&\le
\frac{1}{\alpha}
\sum_{\boldsymbol{x}\in\F_q[t]^k}|F(\boldsymbol{x})|,
\end{align*}
where the last inequality uses the pairwise disjointness of the selected translates. Taking the supremum over all finite $T\subseteq E_\alpha$ gives the desired weak type $(1,1)$ estimate.
\end{proof}

Combining the weak type $(1,1)$ estimate with the strong type $(\infty,\infty)$ estimate and applying Marcinkiewicz interpolation gives
\[
\left\|\mathcal{L}^\ast_sF\right\|_{\ell^2(\F_q[t]^k)}
\ll
\|F\|_{\ell^2(\F_q[t]^k)}.
\]
This proves \eqref{linearmax}.

\section{The ``oscillatory integral'' estimate}

We begin by recalling the notation from Lê-Liu-Wooley \cite{LeLiuWooley2025}.
Let $\mathbb{K}_{\infty}:=\F_q((1/t))$, and let $\mathcal{K}\subset \N$ be a finite set of
positive integers. Given $j,r\in \N$, write
\[
j=\sum_{\nu\ge 0}j_\nu p^\nu,
\qquad
r=\sum_{\nu\ge 0}r_\nu p^\nu
\]
for their base $p$ expansions, where $0\le j_\nu,r_\nu\le p-1$ and all but finitely many
digits are zero. By Lucas' theorem,
\[
\binom{r}{j}\equiv \prod_{\nu\ge 0}\binom{r_\nu}{j_\nu}\pmod p.
\]
Thus $p\nmid \binom{r}{j}$ if and only if $j_\nu\le r_\nu$ for every $\nu$, with the convention that $\binom{r_\nu}{j_\nu}=0$ when $j>r$. We write
\[
j\preceq_p r
\]
when this condition holds. Equivalently, $j\preceq_p r$ when $p\nmid \binom{r}{j}$. This relation
defines a partial ordering on $\N$. In particular, if $j\preceq_p r$, then $j\le r$.

The shadow of $\mathcal{K}$ is defined by
\[
\mathcal{S}(\mathcal{K})
:=
\{j\in \N:j\preceq_p r\text{ for some }r\in \mathcal{K}\}.
\]
Equivalently,
\[
\mathcal{S}(\mathcal{K})
=
\left\{j\in \N:p\nmid \binom{r}{j}\text{ for some }r\in \mathcal{K}\right\},
\]
 We also define
\[
\mathcal{K}^{\ast}
:=
\{k\in\mathcal{K}:p\nmid k\text{ and }p^\nu k\notin \mathcal{S}(\mathcal{K})
\text{ for every }\nu\in\N\}.
\]
L\^e-Liu-Wooley then obtained the following exponential sum estimate:

\begin{thm}[Weyl estimate, {\cite[Theorem~3.1]{LeLiuWooley2025}}]\label{thm:llw}
Fix a set $\mathcal{K}=\{r_1,\ldots,r_k\}\subset \mathbb{N}$. There exist positive constants $c_0$ and $C_0$,
depending only on $\mathcal{K}$ and $q$, such that the following holds. Let $\varepsilon >0$ be arbitrary. Let $n$ be sufficiently
large in terms of $\mathcal{K}, q$ and $\varepsilon$. Suppose that
\[
P(f)=\alpha_1f^{r_1}+\cdots+\alpha_kf^{r_k}
\]
is a polynomial with coefficients in $\mathbb{K}_{\infty}$ satisfying the bound
\[
\left|\sum_{\deg f<n}e(\alpha_1f^{r_1}+\cdots+\alpha_kf^{r_k})\right|\ge q^{n-\eta},
\]
for some $0 \leq \eta\le c_0n$. Then, for each $r_i\in\mathcal{K}^{\ast}$
which is maximal in $\mathcal{K}$ with respect to $\preceq_p$, there exist $a\in\F_q[t]$ and monic $g\in\F_q[t]$ such that $(a,g)=1$,
\[
\ord(g\alpha_i-a)<-r_in+ \varepsilon n + C_0\eta
\qquad\text{and}\qquad
\deg g\le \varepsilon n+ C_0\eta.
\]
\end{thm}
We remark that \cite[Theorem~3.1]{LeLiuWooley2025} is only stated for $\eta>0$, but the proof also works for $\eta=0$. Furthermore, it is easy to see that the result for $\eta>0$ also implies the result for $\eta=0$, by letting $\eta \rightarrow 0$.

It turns out that Theorem \ref{thm:llw} controls
$\wh{M}_n(\boldsymbol{\beta})$ with a lemma which plays the role of the classical oscillatory integral estimates over $\Z$. We call this an oscillatory integral in order to emphasize the analogy with the integer case, but, of course, all sums appearing below are discrete. This is the key input for the proof of the $\varepsilon$-free estimates in the section.

\begin{prop}[Oscillatory integral estimate]\label{smallbeta_inverse}
Consider a finite, nonempty subset $\mathcal{K}=\{r_1,\ldots,r_k\}\subset \N$. Let $n\in \N$, and suppose
$\beta_1,\ldots,\beta_k\in\mathbb{T}$. Suppose $r_k\in \mathcal{K}^\ast$ and that $r_k$ is maximal in $\mathcal{K}$ with respect to $\preceq_p$. 
Define
\[
\Delta:=\ord \beta_k+r_kn,
\]
and suppose that $\Delta<n$. Then one has the pointwise estimate
\begin{equation}\label{smallbeta_inverse_bound}
\left|\sum_{\deg f<n}e(\beta_1f^{r_1}+\cdots+\beta_kf^{r_k})\right|
 \ll_{\mathcal{K},q}
q^{n-c\Delta},
\end{equation}
where $c>0$ is a constant depending on $q$ and $\mathcal{K}$. 
\end{prop}

\begin{proof}
If $\Delta \leq 0$, then the desired estimate follows from the trivial bound
\[
\left|\sum_{\deg f<n}e(\beta_1f^{r_1}+\cdots+\beta_kf^{r_k})\right|\le q^n.
\]
Thus we may assume that $0<\Delta<n$.

If $r_k=1$, then maximality of $1$ and the assumption $1\in\mathcal{K}^\ast$ force $\mathcal{K}=\{1\}$. Since $0<\Delta<n$, by orthogonality we have
\[
\sum_{\deg f<n}e(\beta_1f)=0.
\]
Thus we may assume that $r_k\ge 2$.

Define
\[
m:=\left\lfloor\frac{2\Delta}{r_k}\right\rfloor.
\]
If $m$ is bounded in terms of $\mathcal{K}$ and $q$, then the desired estimate follows by enlarging the implicit constant in \eqref{smallbeta_inverse_bound}. Thus we may assume that $m$ is sufficiently large.

Let $\varepsilon>0$ be sufficiently small ($\varepsilon= 1/3$ will suffice).
We now apply Theorem \ref{thm:llw} with the index set $\mathcal{S}(\mathcal{K})$ and let $c_0,C_0>0$ be the resulting constants. As noted above, we may assume $m$ to be large enough that Theorem \ref{thm:llw} applies with with $m$ in place of $n$.

Choose $\eta\ge 0$ so that
\[
q^{n-\eta}
=\left|\sum_{\deg f<n}e(\beta_1f^{r_1}+\cdots+\beta_kf^{r_k})\right|.
\]
Every polynomial $f$ with $\deg f<n$ can be written uniquely in the form
\[
f=z+t^{n-m}y,
\]
with $\deg z<n-m$ and $\deg y<m$. Hence
\[
q^{n-\eta}
=
\left|
\sum_{\deg z<n-m}\sum_{\deg y<m}
e\left(\sum_{j=1}^k\beta_j(z+t^{n-m}y)^{r_j}\right)
\right|.
\]
By the triangle inequality and pigeonholing, there exists a fixed $z\in\F_q[t]$ with
$\deg z<n-m$ such that

\begin{equation}\label{blocksum}
\left|
\sum_{\deg y<m}
e\left(\sum_{j=1}^k\beta_j(z+t^{n-m}y)^{r_j}\right)
\right|
\ge q^{m-\eta}.
\end{equation}
By expanding
\[
\sum_{j=1}^k\beta_j(z+t^{n-m}y)^{r_j}
\]
via the binomial theorem and collecting terms, we obtain a polynomial in $y$ whose exponents lie in
$\mathcal{S}(\mathcal{K})$. Indeed, the terms which survive in characteristic $p$ are precisely
those for which the corresponding binomial coefficient is not divisible by $p$, and hence their
exponents are in the $\mathcal{S}(\mathcal{K})$. Since $r_k$ is maximal in $\mathcal{K}$, it is also maximal in
$\mathcal{S}(\mathcal{K})$. This forces the coefficient of $y^{r_k}$ to be 
\[
\alpha:=\beta_k t^{r_k(n-m)}.
\]

We now show that $\eta\gg_{\mathcal{K},q}\Delta$. Let
\[
\kappa: = \min(c_0, \frac{\varepsilon}{C_0}).
\]
If $\eta\ge \kappa m$, then, since $m\gg_{\mathcal{K},q}\Delta$ for the range under consideration, we
immediately obtain $\eta\gg_{\mathcal{K},q}\Delta$. Thus we may assume that
\[
\eta<\kappa m.
\]

Since $0\leq \eta<c_0m$, Theorem \ref{thm:llw} applies to the sum appearing in \eqref{blocksum}. Here we also use the fact that we have $r_k\in \mathcal{S}(\mathcal{K})^\ast$ and $r_k$ maximal in $\mathcal{S}(\mathcal{K})$, which is an immediate consequence of our assumptions. 
Hence there exist $a\in\F_q[t]$ and monic $g\in\F_q[t]$ such that
\[
\ord(g\alpha-a)<-r_km+\varepsilon m+C_0\eta \leq -r_km + 2\varepsilon m 
\]
and
\[
\deg g\le \varepsilon m+C_0\eta \leq 2 \varepsilon m,
\]
by the assumtion that  $\eta< \kappa m\le \frac{\varepsilon}{C_0} m$.
In particular, we have $\ord(g\alpha-a)<0$. On the other hand,
\[
\ord(\alpha)=\ord(\beta_k)+r_k(n-m)=\Delta-r_km.
\]
Consequently,
\[\ord(g\alpha) = \deg g + \ord(\alpha) \leq \Delta - r_k m+ 2 \epsilon m < 0,
\]
since $\Delta \leq \frac{r_k}{2} (m+1) $.
This forces $a=0$. Thus,
\[
\deg g+\Delta-r_km
=\ord(g\alpha)
<-r_km+\varepsilon m+C_0\eta,
\]
and hence
\[
\Delta<\varepsilon m+C_0\eta.
\]
Since $m\le 2\Delta/r_k$ and $r_k\ge 2$, this gives
\[
\Delta<\varepsilon \Delta+C_0\eta.
\]
Therefore by rearranging we obtain
\[
\eta\ge \frac{1-\varepsilon }{C_0}\Delta\gg_{\mathcal{K},q}\Delta.
\]
Combining the two cases, we have shown that $\eta\ge c\Delta$ for some constant
$c>0$ depending only on $\mathcal{K}$ and $q$. Therefore
\[
\left|\sum_{\deg f<n}e(\beta_1f^{r_1}+\cdots+\beta_kf^{r_k})\right|
=q^{n-\eta}
\le q^{n-c\Delta}.
\]
This proves the lemma.
\end{proof}

\begin{remark}
One might wonder why we split the exponential modulo $t^{n-m}$ and then applied Theorem \ref{thm:llw} to an exponential sum over $\deg f<m$. Naturally one would first try applying Theorem \ref{thm:llw} to the original exponential sum over $\deg f<n$, but this would lead to a bound of the form $q^{n+\varepsilon n-c\Delta}$, and the $\epsilon n$ loss in the exponent renders this estimate useless for our purposes when $\Delta$ is small. To ameliorate the issue we instead applying Theorem \ref{thm:llw} at scale $m$, with $m$ proportional to $\Delta$.
\end{remark}

\section{The $\varepsilon$-free estimate}
 Using the Proposition \ref{smallbeta_inverse} from the previous section, we prove an $\varepsilon$-free version of Theorem \ref{thm:llw}, following the proof of Wooley \cite[Theorem~1.2]{Wooley2003Freeman} as suggested by Lê-Liu-Wooley directly below the statement of \cite[Theorem~3.1]{LeLiuWooley2025}. This will result in an $\varepsilon$-free version of \cite[Theorem~4.3]{ChampagneGeLeLiuWooley2025} and a more general oscillatory integral estimate. 

\begin{thm}[$\varepsilon$-free Weyl estimate] \label{epsilonfree31}
Fix a set $\mathcal{K}=\{r_1,\ldots,r_k\}\subset \mathbb{N}$. There exist positive constants $c,C,D$,
depending only on $\mathcal{K}$ and $q$, such that the following holds. Let $n$ be sufficiently
large in terms of $\mathcal{K}$ and $q$. Suppose that
\[
P(f)=\alpha_1f^{r_1}+\cdots+\alpha_kf^{r_k}
\]
is a polynomial with coefficients in $\mathbb{K}_{\infty}$ satisfying the bound
\[
\left|\sum_{\deg f<n}e(\alpha_1f^{r_1}+\cdots+\alpha_kf^{r_k})\right|\ge q^{n-\eta},
\]
for some  $0 \leq \eta\le cn$. Then, for each $r_i\in\mathcal{K}^{\ast}$
which is maximal in $\mathcal{K}$, there exist $a\in\F_q[t]$ and monic $g\in\F_q[t]$ such that
\[
\ord(g\alpha_i-a)<-r_in+C\eta + D
\qquad\text{and}\qquad
\deg g\le C\eta + D.
\]
\end{thm}

\begin{proof}
Suppose $r_i\in \mathcal{K}^\ast$ and that $r_i$ is maximal in $\mathcal{K}$. Write
\[
H(f)=P(f)-\alpha_if^{r_i}.
\]
Choose $\varepsilon>0$ sufficiently small ($\varepsilon=1/3$ will suffice). Applying 
Theorem \ref{thm:llw}  with this fixed value of $\varepsilon$, and assuming $n$ is
large enough with respect to $\mathcal{K}$ and $q$ and $\eta<c_0n$, we obtain
$a\in\F_q[t]$ and monic $g\in\F_q[t]$ such that $(a,g)=1$,
\[
\ord(g\alpha_i-a)<-r_in+\varepsilon n+C_0\eta
\qquad\text{and}\qquad
\deg g\le \varepsilon n+C_0\eta,
\]
where $C_0,c_0>0$ are constants depending on $\mathcal{K}$, $q$. Set
\[
d:=\deg g,
\qquad
\beta:=\alpha_i-\frac{a}{g}.
\]
Define
\[
\Delta:=\ord\beta+r_in.
\]
Since $\ord(g\beta)=d+\ord\beta$, we have
\[
d+\Delta<\varepsilon n+C_0\eta < n,
\]
by decreasing $c_0$ if necessary.

We now foliate modulo $g$. Every $f$ with $\deg f<n$ can be written uniquely as
\[
f=gz+y,
\qquad
\deg y<d,
\qquad
\deg z<n-d.
\]
Thus
\begin{equation}\label{foliate}
\sum_{\deg f<n} e(P(f))
=
\sum_{\deg y<d}
e\!\left(\frac{ay^{r_i}}{g}\right)
\sum_{\deg z<n-d}
e\!\left(H(gz+y)+\beta(gz+y)^{r_i}\right).
\end{equation}
For fixed $y$, the inner phase is a polynomial in $z$ whose exponents lie in
$\mathcal{S}(\mathcal{K})$. Since $r_i$ is maximal, the coefficient of $z^{r_i}$ is $\beta g^{r_i}.$ Moreover,
\[
\ord\bigl(\beta g^{r_i}\bigr)+r_i(n-d)
=\ord\beta+r_in
=\Delta.
\]
Proposition~\ref{smallbeta_inverse}, applied at scale $n-d$ with the
exponent set $\mathcal{S}(\mathcal{K})$, gives
\[
\left|
\sum_{\deg z<n-d}
e\!\left(H(gz+y)+\beta(g'z+y)^{r_i}\right)
\right|
\ll_{\mathcal{K},q} q^{n-d-\gamma_1\Delta}
\]
for some $\gamma_1>0$ depending only on $\mathcal{K}$ and $q$. Here we use that
$r_i\in\mathcal{S}(\mathcal{K})^\ast$ and remains maximal in $\mathcal{S}(\mathcal{K})$.
Therefore, using the assumed lower bound and \eqref{foliate},
\[
q^{n-\eta}
\ll_{\mathcal{K},q}
q^d q^{n-d-\gamma_1\Delta}
=
q^{n-\gamma_1\Delta}.
\]
Hence $ \Delta\le \gamma_1\eta + D_1$ for some constant $D_1$ depending only on $\mathcal{K}$ and $q$.

It remains to control $d=\deg g$. We now swap the order of summation in \eqref{foliate}. Since
$d+\Delta<n$, all $y$-dependent terms coming from $\beta(gz+y)^{r_i}$ have order at most 
\[
\ord( \beta (gz)^{r_i-1}  y) \leq \ord(\beta) + (r_i-1)(n-1) + (d-1) = \ord(\beta) +r_i n - r_i -n +d = \Delta +d - n -r_i
< -1 \]
 and
therefore do not affect the character. Thus, we find that
\begin{equation}\label{foliate2}
 \sum_{\deg f<n} e(P(f)) 
=
\sum_{\deg z < n-d} e\left( \beta (gz)^{r_i}\right)
\sum_{\deg y<d}
e\!\left(H(gz + y) + \frac{ ay^{r_i}}{g}\right).
\end{equation}
For fixed $z$,
the exponents appearing in the inner sum over $y$ again lie in $\mathcal{S}(\mathcal{K})$, and by maximality of $r_i$
the coefficient of $y^{r_i}$ is exactly $a/g$. 
Let $c_1, C_1 >0$ be the constants depending only on $\mathcal{K}$ and $q$ given by Theorem \ref{thm:llw} when applied to the exponent set $\mathcal{S}(\mathcal{K})$. We may further assume that $d$ is sufficiently large in terms of $\mathcal{K}$ and $q$, so that Theorem \ref{thm:llw} applies at scale $d$ with our fixed choice of $\varepsilon$. 
We claim that there is a constant
$\gamma_2>0$, depending only on $\mathcal{K}$ and $q$, such that uniformly in $z$,
\[
\left|
\sum_{\deg y<d}
e\!\left(H(gz+y)+\frac{a}{g}y^{r_i}\right)
\right|
\leq q^{d-\gamma_2 d}.
\]
Indeed, if this failed, then Theorem \ref{thm:llw} applied at scale $d$ would give $b\in\F_q[t]$ and monic $\ell\in\F_q[t]$ such that
\[
\ord\!\left(\ell\frac{a}{g}-b\right)
<
-r_id+\varepsilon d+C_1\gamma_2 d
\qquad\text{and}\qquad
\deg \ell\le \varepsilon d+C_1\gamma_2 d.
\]
Choosing $\varepsilon$ and $\gamma_2$ sufficiently small gives $\deg \ell<d$ and, when $r_i\ge 2$,
\[
-r_id+\varepsilon d+C_1\gamma d<-d.
\]
But since $(a,g)=1$ and $\deg g=d$, we have that
$\ell a/g-b$ is nonzero and has order at least $-d$, a contradiction. The case $r_i=1$ is
handled directly by orthogonality in the nontrivial linear coefficient $a/g$. This proves the
claim.

Consequently,
\[
\left|\sum_{\deg f<n} e(P(f))\right|
\leq
q^{n-d}q^{d-\gamma_2 d}
=
q^{n-\gamma_2 d}.
\]
Combining this with the lower bound $q^{n-\eta}$ gives $d\le \gamma_2\eta$. Thus in all cases of $d$, we have
\[
\deg g = d \leq \gamma_2 \eta + D_2
\]
for some constant $D_2>0$ depending only on $\mathcal{K}$ and $q$.

Finally,
\[
\ord(g\alpha_i-a)
=
d+\ord\beta
=
-r_in+d+\Delta
\le -r_in+(\gamma_1+\gamma_2)\eta + (D_1 + D_2).
\]
By letting $C = \gamma_1 + \gamma_2$ and $D=D_1+D_2$, This proves the theorem.
\end{proof}

Now we may prove the following.

\begin{cor}\label{epsilonfree43}
Fix a set $\mathcal{K}=\{r_1,\ldots,r_k\}\subset \mathbb{N}$. Suppose further that $(r_i,p)=1$ for every $1\leq i\leq k$. There exist positive constants
$c,C,D$, depending only on $\mathcal{K}$ and $q$, such that the following holds. Let
$n$ be sufficiently large in terms of $\mathcal{K}$ and $q$. Suppose that
\[
P(f)=\alpha_1f^{r_1}+\cdots+\alpha_kf^{r_k}
\]
is a polynomial with coefficients in $\mathbb{K}_{\infty}$ satisfying the bound
\[
\left|\sum_{\deg f<n}e(\alpha_1f^{r_1}+\cdots+\alpha_kf^{r_k})\right|\ge q^{n-\eta},
\]
for some $0 \leq \eta\le cn$. Then, for any $1\leq i\leq k$, there exist 
$a_i\in\F_q[t]$ and monic $g_i\in\F_q[t]$ such that
\[
\ord(g_i\alpha_i-a_i)<-r_in+C\eta + D
\qquad\text{and}\qquad
\deg g_i\le C\eta + D.
\]
\end{cor}

\begin{proof}
We proceed by downward induction. Suppose without loss of generality that $r_1<\cdots<r_k$.
The case $i=k$ follows immediately from Theorem \ref{epsilonfree31}, since $r_k$ is maximal in
$\mathcal{K}$ and, because $(r_k,p)=1$, we have $r_k\in\mathcal{K}^{\ast}$.

Now suppose the result holds for $j<i\leq k$. We prove it for $r_j$. For each $j<i\leq k$,
let $c_i,C_i,D_i>0$ be the constants appearing in the induction hypothesis, and choose $c>0$
sufficiently small in terms of these constants and $\mathcal{K},q$. Applying the induction
hypothesis to each $i>j$, we obtain $a_i\in\F_q[t]$ and monic $g_i\in\F_q[t]$ such that
\[
\ord(g_i\alpha_i-a_i)<-r_in+C_i\eta + D_i
\qquad\text{and}\qquad
\deg g_i\le C_i\eta + D_i.
\]

For $j<i\leq k$ define
\[
g:=g_{j+1}\cdots g_k
\qquad\text{and}\qquad
b_i:=a_i\prod_{\substack{j<\ell\le k\\ \ell\neq i}}g_\ell.
\]
Then $g$ is monic, and for every $i>j$ we have
\[
\ord (g\alpha_i-b_i)<-r_in+C_1\eta+D^\ast
\qquad\text{and}\qquad
\deg g\le C_1\eta+D^\ast,
\]
where $C_1, D^\ast>0$ depends only on the constants already obtained in the induction.

Choose $L>0$ sufficiently large in terms of $\mathcal{K},q$ and $C_1$, and set
\[
M=\lfloor n-L\eta-D^\ast-2\rfloor .
\]
After decreasing $c$ if necessary, we may assume that $M$ is sufficiently large and that
$M+\deg g<n$. As in the proof of \cite[Theorem~4.1]{LeLiuWooley2025}, we partition
$\{f\in\F_q[t]:\deg f<n\}$ into $q^{n-M}$ blocks of the form
\[
B_s=\{gy+s:\deg y<M\}.
\]
By the pigeonhole principle, there exists a block $B_s$ such that
\[
\left|\sum_{\deg y<M} e(P(gy+s))\right|\ge q^{M-\eta}.
\]
By the choice of $L$ and the estimates for $g\alpha_i-b_i$ with $i>j$, the same argument used to
obtain \cite[(4.7)]{LeLiuWooley2025} shows that
\[
\left|
\sum_{\deg y<M}
e\bigl(\alpha_1(gy+s)^{r_1}+\cdots+\alpha_j(gy+s)^{r_j}\bigr)
\right|
\ge q^{M-\eta}.
\]

The coefficients in the expansion of
\[
\alpha_1(gy+s)^{r_1}+\cdots+\alpha_j(gy+s)^{r_j}
\]
as a polynomial in $y$ have exponents lying in $\mathcal{S}(\mathcal{H}_1)$, where
\[
\mathcal{H}_1:=\{r_1,\ldots,r_j\}.
\]
Moreover, since $r_j$ is coprime to $p$ and is the maximum of $\mathcal{H}_1$, we have
$r_j\in \mathcal{S}(\mathcal{H}_1)^\ast$, and $r_j$ is maximal in $\mathcal{S}(\mathcal{H}_1)$.
The coefficient of $y^{r_j}$ in the expansion is $\alpha_jg^{r_j}$. Therefore, by Theorem
\ref{epsilonfree31}, applied at scale $M$, there exist $\widetilde{a}_j\in\F_q[t]$ and monic
$\widetilde{g}_j\in\F_q[t]$ such that
\[
\ord(\widetilde{g}_j\alpha_jg^{r_j}-\widetilde{a}_j)<-r_jM+C_j'\eta+D'
\qquad
\text{and}
\qquad
\deg \widetilde{g}_j\le C_j'\eta+D',
\]
where $C_j',D'>0$ depends only on $\mathcal{K}$ and $q$.

Finally, take
\[
g_j:=\widetilde{g}_jg^{r_j}
\qquad\text{and}\qquad
a_j:=\widetilde{a}_j.
\]
Then
\[
\ord(g_j\alpha_j-a_j)<-r_jM+C_j'\eta+D'
<-r_jn+C_j\eta+D_j
\]
and
\[
\deg g_j
\le \deg\widetilde{g}_j+r_j\deg g
\le (C_j'+r_jC_1)\eta+D'+r_jD^\ast
\le C_j\eta+D_j,
\]
after choosing $C_j, D_j$ appropriately. This proves the induction step, and hence the corollary.
\end{proof}

From this, we obtain the following useful corollary, which serves as the oscillatory integral estimate in the proof of the main theorem.

\begin{cor}\label{smallbeta_inverse_ultimate}
Fix a set $\mathcal{K}=\{r_1,\ldots,r_k\}\subset \N$. Let $n\in \N$, and suppose
$\beta_1,\ldots,\beta_k\in\mathbb{T}$. Suppose $(r_i,p)=1$ for every $1\leq i\leq k$. 
Define
\[
\Delta_i:=\ord \beta_i+r_in,
\]
and suppose that $ \Delta_i<n/2$. Then one has the pointwise estimate
\begin{equation}\label{smallbeta_inverse_ultimate_bound}
\bigl|\sum_{\deg f<n}e(\beta_1f^{r_1}+\cdots+\beta_kf^{r_k})\bigr|
\ll_{\mathcal{K},q}
q^{n-c\Delta_i},
\end{equation}
where $c>0$ is a constant depending on $q$ and $\mathcal{K}$.
\end{cor}

\begin{proof}
The case $\Delta_i \leq 0$ follows from the trivial bound, so assume $\Delta_i>0$. It is enough to
prove the result for $n$ sufficiently large in terms of $\mathcal{K}$ and $q$, since the remaining
values of $n$ can be handled by increasing the implicit constant.

Let $c_0,C_0,D_0>0$ depending on $\mathcal{K}$ and $q$ be the constants in Corollary
\ref{epsilonfree43}. Let $\kappa>0$ be a small constant depending only on
$\mathcal{K}$ and $q$ to be chosen later. If the sum is zero, there is nothing to prove. Otherwise,
choose $\eta\ge 0$ so that
\[
q^{n-\eta}
=
\left|
\sum_{\deg f<n}e(\beta_1f^{r_1}+\cdots+\beta_kf^{r_k})
\right|.
\]
If $\eta\ge \kappa\Delta_i$, then the desired estimate follows. Thus we may assume
\[
\eta<\kappa\Delta_i.
\]
Assume for the moment that we also have $\eta>0$. 
Since $\Delta_i<n/2$, after decreasing $\kappa$ if necessary we have
$\eta<c_0n$, and hence Corollary \ref{epsilonfree43} applies with target exponent
$r_i$. Therefore there exist $a\in\F_q[t]$ and monic $g\in\F_q[t]$ such that
\[
\ord(g\beta_i-a)<-r_in+C_0\eta+D_0
\qquad\text{and}\qquad
\deg g\le C_0\eta+D_0.
\]
Now
\[
\ord(g\beta_i)
=
\deg g+\ord\beta_i
\le C_0\eta+D_0-r_in+\Delta_i
< -r_in+D_0+(1+C_0\kappa)\Delta_i.
\]
Since $\Delta_i<n/2$, this quantity is negative as long as $\kappa$ is sufficiently small with
respect to $C_0$ and $\mathcal{K}$. Also,
\[
-r_in+C_0\eta+D_0<0
\]
after decreasing $\kappa$ once more. Therefore both $\ord(g\beta_i)$ and
$\ord(g\beta_i-a)$ are negative. This forces $a=0$, since if $a\neq 0$, then
$a\in\F_q[t]$ has nonnegative order and the ultrametric inequality gives
$\ord(g\beta_i-a)\ge 0$.

Consequently,
\[
\ord(g\beta_i)<-r_in+C_0\eta+D_0.
\]
Using
\[
\ord(g\beta_i)=\deg g+\ord\beta_i=\deg g-r_in+\Delta_i,
\]
we obtain
\[
\deg g+\Delta_i<C_0\eta+D_0.
\]
In particular,
\[
\Delta_i<C_0\eta+D_0.
\]
Combining this with the first case $\eta\ge \kappa\Delta_i$, and taking
$c=\min\{\kappa,1/C_0\}$ after decreasing it if necessary, we have
\[
\left|
\sum_{\deg f<n}e(\beta_1f^{r_1}+\cdots+\beta_kf^{r_k})
\right|
=q^{n-\eta}
\ll_{\mathcal{K},q}q^{n-c\Delta_i},
\]
as needed.

Finally, if $\eta=0$, then applying the above argument with $0<\eta'<\kappa \Delta_i$ yields $\Delta_i\ll_{\mathcal{K},q} \eta'+D_0$. Taking $\eta'\to 0$ yields $\Delta_i\ll_{\mathcal{K},q}1$, which trivially implies the desired result. 
\end{proof}

Next we bound the Gauss sum $\Lambda(a_1,...,a_k,h)$. This follows from the proof of
\cite[Lemma~7.5]{LeLiuWooley2025}, with Corollary \ref{epsilonfree43} used in place of
\cite[Theorem~4.1]{LeLiuWooley2025}.
\begin{cor}[Gauss sum bound]\label{GaussSumBound}
Fix a finite set $\mathcal{K}=\{r_1,\ldots,r_k\}\subset\mathbb{N}$ with $(r_i,p)=1$ for every $1\leq i\leq k$. Then there exists a constant $\gamma>0$, depending only on $\mathcal{K}$ and $q$, such that the following holds. If $h\in \F_q[t]$ is monic and $a_1,\ldots,a_k\in \F_q[t]$ satisfy $(a_1,\ldots,a_k,h)=1$, then
\[
|\Lambda(a_1,\ldots,a_k,h)|\ll_{\mathcal{K},q} q^{-\gamma\deg h}.
\]
\end{cor}

\begin{proof}
If $\deg h=0$, then the claim is trivial. Thus we may assume that $m:=\deg h\geq 1$. For each
$1\leq i\leq k$, write
\[
d_i:=(a_i,h)
\qquad\text{and}\qquad
h_i:=h/d_i.
\]
Since each $r_i$ is coprime to $p$, the proof of \cite[Lemma~7.5]{LeLiuWooley2025},
using Corollary \ref{epsilonfree43} in place of \cite[Theorem~4.1]{LeLiuWooley2025},
gives the following estimate. For each $1\leq i\leq k$ there exists a constant $C_i>1$,
depending only on $r_i$, $\mathcal{K}$, and $q$, such that for every $\varepsilon>0$,
\begin{equation}\label{GaussInit}
\sum_{\deg f<m} e\!\left(\frac{a_1f^{r_1}+\cdots+a_kf^{r_k}}{h}\right)
\ll_{\mathcal{K},\varepsilon,q}
q^{\deg d_i/C_i}q^{m(1-1/C_i+\varepsilon)}
\end{equation}
whenever $a_i\neq 0$.

Set
\[
C:=\max_{1\leq i\leq k} C_i
\qquad\text{and}\qquad
\varepsilon:=\frac{1}{2kC}.
\]
Since $(a_1,\ldots,a_k,h)=1$, we have
\[
\operatorname{lcm}(h_1,\ldots,h_k)=h.
\]
Hence
\[
\sum_{i=1}^k \deg h_i\geq \deg h=m.
\]
By the pigeonhole principle, there exists $i_0$ such that
\[
\deg h_{i_0}\geq \frac{m}{k}.
\]
In particular $a_{i_0}\neq 0$, so we may apply \eqref{GaussInit} with index $i_0$. Dividing by
$q^m$, we obtain
\begin{align*}
|\Lambda(a_1,\ldots,a_k,h)|
&=q^{-m}\left|\sum_{\deg f<m} e\!\left(\frac{a_1f^{r_1}+\cdots+a_kf^{r_k}}{h}\right)\right|\\
&\ll_{\mathcal{K},q} q^{-m}q^{\deg d_{i_0}/C_{i_0}}q^{m(1-1/C_{i_0}+\varepsilon)}\\
&=q^{-\deg h_{i_0}/C_{i_0}+\varepsilon m}\\
&\leq q^{-m/(kC_{i_0})+\varepsilon m}\\
&\leq q^{-m/(kC)+\varepsilon m}\\
&=q^{-m/(2kC)}.
\end{align*}
Taking
\[
\gamma:=\frac{1}{2kC}
\]
proves the proposition.
\end{proof}

\nocite{Bergelson2006,LeLiuWooley2025,ChampagneGeLeLiuWooley2025}
\bibliographystyle{alpha}
\bibliography{references}

\end{document}